\documentclass[12pt]{article}
\usepackage{amsmath,amssymb,amsfonts,amsthm}
\newenvironment{myabstract}{\par\noindent
{\bf Abstract . } \small }
{\par\vskip8pt minus3pt\rm}
\newcounter{item}[section]
\newcounter{kirshr}
\newcounter{kirsha}
\newcounter{kirshb}
\newenvironment{enumroman}{\setcounter{kirshr}{1}
\begin{list}{(\roman{kirshr})}{\usecounter{kirshr}} }{\end{list}}
\newenvironment{enumarab}{\setcounter{kirshb}{1}
\begin{list}{(\arabic{kirshb})}{\usecounter{kirshb}} }{\end{list}}
\newtheorem{theorem}{Theorem}[section]

\newtheorem{corollary}[theorem]{Corollary}
\newenvironment{demo}[1]{\noindent{\bf #1.}\upshape\mdseries}
{\nopagebreak{\hfill\rule{2mm}{2mm}\nopagebreak}\par\normalfont}
\theoremstyle{definition}

\newtheorem{example}[theorem]{Example}
\newtheorem{definition}[theorem]{Definition}

\def\C{{\mathfrak{C}}}
\def\Fm{{\mathfrak{Fm}}}

\def\Nr{{\mathfrak{Nr}}}

\def\Fm{{\mathfrak{Fm}}}
\def\A{{\mathfrak{A}}}
\def\B{{\mathfrak{B}}}
\def\C{{\mathfrak{C}}}
\def\D{{\mathfrak{D}}}

\def\Sn{{\mathfrak{Sn}}}
\def\CA{{\bf CA}}

\def\K{{\bf K}}
\def\K{{\bf K}}

\def\(R)RA{{\bf (R)RA}}

\def\Zd{{\mathfrak Zd}}
\def\Co{{\sf Co}}
\def\Ig{{\sf Ig}}

 \def\CA{{\sf CA}}
\def\B{{\sf B}}
\def\G{{\sf G}}

\def\K{{\sf K}}

\def\Nr{{\mathfrak{Nr}}}

\def\Nr{{\mathfrak{Nr}}}

\def\A{{\mathfrak{A}}}
\def\B{{\mathfrak{B}}}
\def\C{{\mathfrak{C}}}
\def\D{{\mathfrak{D}}}

\def\A{{\mathfrak{A}}}
\def\B{{\mathfrak{B}}}
\def\C{{\mathfrak{C}}}
\def\D{{\mathfrak{D}}}

\def\L{{\mathfrak{L}}}

\def\L{{\mathfrak{L}}}
\def\CA{{\bf CA}}

\def\G{{\bf G}}

\def\O{{\mathfrak{O}}}

\title{Back and forth between algebraic geometry, algebraic logic, sheaves, and forcing}
\author{Tarek Sayed Ahmed}

\begin{document}
\maketitle

\begin{myabstract} We study sheaves in the context of a duality theory for lattice structure endowed with extra operations, 
and in the context of forcing in a topos.
Using Sheaf duality theory of Comer for cylindric algebras, 
we give a representation theorem of of distributive bounded lattices expanded by modalities (functions distributing over joins)
as the continuous sections of sheaves. 
Our representation is defined via a contravariant functor from an algebraic category to a category of sheaves. 
We show that if our category is a small site (cartesian closed with a stability condition on pullbacks), 
then we can define a notion of forcing using this category. In particular, we define fuzzy forcing by interpreting
the additional Lukasiewicz conjunction $\otimes$ as induced by a tensor product in the target monodial 
category of pre-sheaves.
We also study topoi as semantics for higher order logic of many sorted theories in connection to set theory, and the quasi-topoi
based on $MV$ algebras, for fuzzy logic. We show that the interpretation of  a theory $T$, in this case
into $Set_{\Omega}$ where $\Omega$ is an almost sub-object classifier in a quasi-topos $CAT$  defined from $T$,
is completed by defining semantics for $\otimes$, and this is done similarly to its defining clause in forcing.
We give applications to many-valued logics  
and  various modifications of  first order logic and multi-modal logic, set in an algebraic framework. 
\end{myabstract}

\section{Introduction}

Algebraic logic is an interdisciplinay field between algebra and logic, in fact 
it is the natural interface between logic and  universal algebra. It is similar in this respect to other branches
of mathematics  like algebraic topology and algebraic geometry.
In the latter case, for example, one addresses geometric problems using algebraic machiney but the underlying 
ideas are guided by geometric intuition. In agebraic logic a similar task is implemented; 
attacking problems in various logical systems,
using algebraic machinery, except that now the intuition is based on ideas stemming from formal logic. 
There is, yet another geometric twist to Tarskian algebraic logic, where representable cylindric algebras can be visualized
as a multi - dimensional geometry, the dimension can be a transfinite ordinal, hence the term 'cylindric', 
which actually means that the geometric set theoretic interpretation of the existential quantifies in first order is nothing more than forming
cylinders, when interpreted in set algebras based on first order models.

In algebraic geometry the concept of a {\it ringed space} is important. 
A ringed space, is basically a site (a cartesian closed category, with additional localization properties), 
or more concretely a topological space, 
where each point in the underlying set of the topology, is associated with a ring, called a germ or a stalk, such that the 
this topological space with an amalgamation of the local rings, form a {\it sheaf}. 

A sheaf   is a tool for systematically tracking locally defined data attached to the open sets 
of the topological space. This  data can be restricted to partitions of open sets into smaller sets, 
and the data assigned to an open set is equivalent to all collections of compatible data assigned to the collection of smaller  sets 
covering the original one. For example, such data can consist of the  rings defined on each such set. 
Sheaves are by design quite general and abstract objects, and their correct definition is rather technical. 
They exist in several varieties, such as sheaves of  sets, or sheaves of rings,
depending on the type of data assigned to open sets. 

In the first  paper we will be primarily concerned with sheaves
associated to algebraizations of a multitude of predicate logics. The topological space on which our sheaves
are based is the prime spectrum of the zero-dimensional 
subreducts of the algebra in question, endowed with the Zariski topology, and the germs 
will be homomorphic images of the algebra, determined by the ideals generated by the prime ideals
of the zero dimensional part.

In all cases we are dealing with sheaves of 'locally algebraised' topological spaces.   

In the second part of the paper, we will be concerned with the notion of pre-sheaves on Grothendieck sites, but
from the logical point of view and not the geometric one.
In more detail, we will study forcing in a topos, which is an abstraction of the category of sets. 
On the hand, an {\it elementary} topos is a category that is closed under familar operations on sets (like finite products),
on the other hand a general 
topoi can be viewed as a functor from a site to a category of pre-sheaves.

Boolean algebras correspond to sets, example via Stone representation theorem, 
and Boolean algebras are also used in forcing by perturbing the ground model, forming Boolean valued
models. Such forcing can be formulated in the context of the topos of sets, but it lends itself to 
many generalizations. For example one can study forcing, by viewing  the comulative 
heirarchy of sets as pre-sheaves defined on Heyting algebras, or sheaves defined on a topological space.

In algebraic geometry, affine varieties also carry a Zarski  topology by declaring the closed sets 
to be to be precisely the affine algebraic sets. 

Given an affine variety $A^n$ over a field $K$, the ring associated with $A^n$ is the ring $K[x_1,\ldots x_n]$, which is the polynomial ring
in $n$ variables over the field $K$. If $V\subseteq {}^nA$, and $I(V)$ is 
the ideal of all functions vanishing on $V$, that is $\{f\in K[x_1,..x_n]: f(s)=0, \forall s\in V\}$, then 
for any affine algebraic set $V$,  the coordinate ring or structure ring of $V$ is the quotient 
of the polynomial ring by this ideal. 

$I(V)$ is a prime ideal in the polynomial ring, and the Zariski topology can be viewed as the topology based
on this prime spectrum. 

This phenomena has a more than one abstract setting. We mention two.

(1) This is a particular case of what is known in the literature of algebraic geometry  as coherent sheaves of modules.
(This is basically obtained by replacing the field $K$ by a ring, and the affine variety by a module).

(2) If $Spec(R)$ is the prime spectrum of a commutative ring $R$, 
then the {\it stalk} at $P$ equals the localization of $R$ at $P$, and this is a local ring.  
Endowed with the Zariski topology, $Spec (R)$,  
commonly augmented also with a sheaf structure, makes it a locally ringed space.

A  typical problem concerning ringed spaces, and in particular affine varities is 
describing the polyonomial ring associated with an affine variety  in terms of the local rings, 
or stalks, given at a point of the variety.

In the first paper part, our points will be theories, or prime ideals in the quantifier free reducts of the Lindenbaum Tarski algebra, or the algebra
of sentences.
Following Comer, we describe the algebra of formulas, corresponding to a given theory as the algebra of continuous sections of a sheaf,
that is, the continuous maps from the prime spectrum of the zero dimensional part of this 
algebra with the Zariski toplogy, to an amalgamation 
of the stalks, endowed with
a natural topology. This theory corresponds to an ideal. 
In our case, the stalks are algebras defined locally at each point of the spectrum, and they can be 
amalgamated by taking ther algebraic product. 
In favourable circumstances, for example in the case of locally finite cylindric algebras,
or quasipolyadic algebras, the stalks turn out to be simple algebras.

Then in this case, what we are actually doing is taking the algebra factored out by the intersection of all maximal ideals 
containing the ideal we started off with.
This algebra turns out {\it naturally 
isomorphic} to the algebra of formulas (in the precise categorial sense).

Our situation also has affinity to the duality of Boolean algebras and Stone spaces, 
in fact, it is a generalization thereof. The Stone duality is obtained from Comer's duality theory
for cylindric algebras, when the stalks are just the two element Boolean algebra.

Expressed in a metalogical setting,
we deal with {\it two algebras} not just one, 
the former is the algebra of formulas, which will have extra operations reflecting quantifiers, the other
is the Boolean algebra of sentences. The Stone space of the algebra of sentences  
is used to define a dual of the algebra of formulas, but this does not capture
the quantifier structure;  the dual we are looking for will actually be a triple,
the Stone space of the algebra of sentences,  a disjoint union of stalks 
(which are homomorphic images of the algebra of formulas, obtained by factoring out this algebra by 
complete extensions of the given theory; every such extension 
gives a stalk), endowed with a natural topology induced, 
by the projection map $\pi$, which projects the stalk at a completion of the theory, to the theory.  

Stone duality thus becomes a special case, because in this case we do not have have quantifiers 
so the algebra of formulas is the same as algebra of sentences.
More rigorously, the stalks are just the two element Boolean algebra, and the duality in this case, reduces to the classical Stone case,
implemented via applying the contravariant Hom functor, with second component a co-separator, twice. 
We also formulate our duality theorem concerning expansions of distributive lattices, as a double application of the Hom 
contravarinat functor.
(This will be further elaborated upon below).

In the concrete case of usual first order logic the stalks are simple algebras, 
so that the disjoint union can be viewed in a natural way as a semisimple 
algebra.  Such algebras are called regular; regular algebras are semisimple, but the converse fails dramatically.
On the other hand, it is not always the case that the stalks are simple algebras nor indeed subdirectly
indecomposable, for other extensions of first order logic.

The idea is taken from Comer, except that here we substantially generalize Comer's duality theory of cylindric algebras.
Comer's results apply to first order logic, particularly, to studying interpolation theorems like Craig interpolation and Beth definabilty.
Our results apply to a plathora of logics, like many valued logics, intuitionistic logic, 
fuzzy logic,  different modifications of first order logic (like finite variable fragments) 
and its extensions like Keislers logics. 

Since Comer dealt with classical logic, the topology used on the algebra of sentences is the Stone topology.
Here we shall deal with algebraic structures, whose dual topology, 
is based on the spectrum of prime ideals (these do not nessecarily coincide with maximal ones), namely, the Zariski topology, 
or ever the weaker Priestly topology. 

To each such structure $\A$, which will be expanded by operators reflecting quantification viewed in the modal sense,
one associates a sheaf, whose first component is a Priestly topology. The latter has underlying set $X$, the set of prime ideals  
of the algebra consisting of zero dimensional elements. These are the elements
that are fixed points of the operators, and indeed form 
a subalgebra of the reduct of the original algebra obtained by discarding the 
modalities. 

Like in the classical case, the second component is both a disjoint union of stalks $\G_x$, $x\in X$ that is endowed
with a topology induced by certain maps, namey the $\sigma_a$s defined below, 
and also can be formulated as an algebraic product.
The third is the projection map.

If one takes the dual of this triple, more concisely the sheaf thereby obtained,  
one can represent the algebra he started off with  
as the algebra of continuous sections of this sheaf. These are maps from $X$ to the disjoint union of stalks $\delta$, and 
the continuity  here is with respect to the Priestly topology on $X$, and the  smallest topology on $\delta$ 
with respect to which the all maps of the form $\sigma_a: X\to \delta$, $\sigma_a(x)=a/\Ig^{\A}x\in \G_x$
are continuous.

So here we are infront of a very natural (natural) isomorphism, or a double dual. 
The representation in this geometric context is implemented by a 
contravariant functor, which describes this process and its inverse. 

One aspect of this duality is that, for example,
there is an isomorphism between the set of certain ideals of $\A$ called
regular ones, onto the lattice of open subsets of $X.$ (This will be proved below). 
A regular ideal is one that gives rise to a simple
stalk.

Furthermore, one can use such duality, to prove theorems formulated in categorial jargon for such algebras,
by working in the dual space of  sheaves. 

A typical example is whether epimorphisms are surjective or not in the given
class,  which is the equivalent (in a very broad context, including multi -dimensonal modal logics)
of whether the logic in question has the Beth definability property, which in turn is equivalent
to whether monomorphisms are injective in the dual space of sheaves. 

Sometimes,  it is much easier to work in the dual world, by turning round arrows and reversing composition; this happens often in 
situations that involve arrows, via morphisms.
A blatant example of such a phenomena is the fact that to prove that epimorphisms are surjective in Boolean algebras, then it is so much
easier to prove that monomorphisms are injective in the category of Stone spaces.

We will do the same here for distributive lattices with extra operators, when we show that if epimorphisms fail to be surjective 
in the class of so called regular algebra then they also  fail to be surjective in the simple algebras. 
Regular algebras are these algebras for which  the lattice of ideals of their zero dimensional part  (the subreduct fixed by the operators)
is isomorphic to the their lattice of ideals. 
Such algebras lie strictly between simple and semi-simple algebras; in algebraic logic they are a common abstraction of locally finite 
algebras and simple algebras.

In the second part of the paper we deal with sheaves in an entirely different context. We will study functors from sites to a category of pre-sheaves.
This general framework covers forcing constructions, both in the clasical and intuitionistic sense. 
In both cases the notion of forcing is a partially ordered set, which  gives rise to a cite.
In the intuitionistic context,  the Heyting algebra obtained consists of open sets of a natural topology on $P$, whihe in the classical sense,
the Boolean algebra is based on the regular open sets on the same
topology. Both are complete.

Both the Heyting and Boolean valued model can be seen as a 
reflective subcategory of a category of pre-sheaves, consisting only of sheaves. 

One can do forcing in a topos, and also one can define semantics of 
higher order logics in 
topoi. 

We will show that these two approaches to higher order logics, which has other re-incarnations in the literature, like type theory and Lambda calculas
are very much related. In fact, we will show that if $F$ is a functor from a small site $C$ 
to a category of pre-sheaves, then
$F$ defines a notion of forcing $C$ and the Yoneda image of $C$, namely, the set of all
 Hom functors on $C$, is the ramified language.  This view will make us define fuzzy forcing by interpreting the Lukasiewicz
conjunction on a monoid $C$, as induced by a tensor product in the target monoidal category.

Finally, we study interpretations of higher order theories in a topoi in the classical sense and in a quasi-topoi in the fuzzy sense.
While the interpretation in the first case, is coded by a Heyting algebra, that is the Tarski Lindenbaum algebra of definable subsets of the 
given theory $T$,
the second interpretation, based on an $MV$ algebra $A$ has no sub-object classifier. But interpreting the Lukasiewicz connection as normal conjunction,
we get a Heyting algebra $H$ that  is an almost sub-object classifier, and the category it generates, namely $V^H$ 
is a subcategory of $V^A$.  
Both the $MV$ algebra, and the Heyting algebra obtained, represent forcing with different sites, 
and using the defining clause for $\otimes$, we can define semantics for the many sorted fuzzy theory, in a quasi-topoi.

\section{Sheaves, back and forth between logic and geometry}

A sheaf is a central concept that occurs in algebraic geometry, and its definition is somewhat technical. 
But roughly a sheaf can be viewed as pre-sheaf with an additional 'glueing' condition, 
and it is best formulated in category theory. 
(A task implemented by the giant  Grothendieck, in the context of algebraic geometry). 

One way to define a pre-sheaf is, to define it as a contravariant functor from a site $X$, or more concretely a topological space $X$ 
to a target concrete category ${\sf C}$. This category ${\sf C}$ can consist of just sets, or groups or rings and sets of maps or whatever.
The objects of $X$ are the open sets and the morphisms are inclusions, while in the target category 
the objects are the concrete objects and the morphisms are the natural ones (in algebras they are the homomorphisms, 
in topological spaces they are
homeomorphisms, and so on).

Let $F$ be a functor defining a pre-sheaf. A {\it stalk} of the pre-sheaf 
is $i^-F(\{x\})$, where $i$ is the inclusion of the one point $\{x\}$ into $X$. 

A sheaf is obtained when one requires, in addition, that the stalks can be `glued', for example if the cite is a topological
space, the stalks are algebras, which is usually
the case, then a glueing, is 'continuously varying' these algebras, and this can be implemented by a 
subdirect product of the algebras, and the continuity is measured  with respect to the smallest
topology on this product that makes a given family of maps from the topological space $X$, 
to the the product, continuous. 
In algebraic geometry the sheaf on an affine variety glueing the local rings
is extracted from the the structure of the variety using the data of the Zariski topology 
and (a subsheaf of the) sheaf of {\it germs}, which is a sheaf of local rings. 

Let $X$ be a topological space, define $O(X)$ whose objects are open sets of $X$ and morphism are inclusions. 
Then a presheaf is a contravariant functor from $O(X)$ to $C$ any category. Keeping the target category as it as, this definition can  be widened 
when the source category does nor arise so concretely from a topololical space.  
If $F$ is a $C$ valued presheaf on $X$, then $U$ is an open set, then  $F(U)$ is called the sections of $F$ over $U$. 
This generalizes when the source category is known as a site. 
We first give the concrete definition, with minimal categorical jargon:

\begin{definition} Let $X$ be a topological space, and let $\bold C$ be a category. A presheaf $F$ on $X$ with values in $\bold C$ is given by:
\begin{itemize}
\item For each open set $U$ of $X$ there corresponds an object $F(U)$ in $C$.
\item For each inclusion $U\subseteq V$ , there corresponds a morphism $rep_{U,V}:F(U)\to F(V)$ in the category $\bold C.$
\end{itemize}. 
These morphisms, are called restriction morphisms. They satisfy the following compatibilty conditons:
\begin{itemize}
\item $res_{U,U}:F(U)\to F(U)$ is the identity
\item For $U\subseteq V\subseteq U$, we have $res_{W,V}\circ res_{V,U}=res_{W,U}$
\end{itemize}

A sheaf is a presheaf that satisfies
\begin{itemize}

\item if $(U_i)$ is a open cover of an open set $U$ and if $s,t\in F(U)$ such that $s\upharpoonright U_i=t\upharpoonright U_i$ 
for each $U_i$ then $s=t$, and

\item  If $(U_i)$ is an open covering of an open set $U$ and for each $i$ there is a section $s_i$ of $F$ over $U_i$ such that for each pair $U_i$,$U_j$
if the covering sets the restrictions of $s_i$ and $s_j$ agree on overlaps: $s_i:\upharpoonright U_i\cap U_j=s_j\upharpoonright U_i\cap U_j$, 
then there is a section $s\in F(U)$ such that $s\upharpoonright U_i=s_i$ for each $i$.
\end{itemize}
\end{definition}

In the above definition, we started with a topogical space. 
In Category theory a Grothendeick topology is a structure on a category $C$ that makes the objects of $C$ 
act like open sets, it need not be a real topology. 
In Grothendeick topology the notion of a collection of open subsets that are stable under  inclusion is replaced by a {\it seive.}
A seive on $C$ is a subfunctor of the functor $Hom(-,C)$ .

The best example to give for sheaves and pre-sheaves, is in the context of forcing formulated in a  topos, 
rather than in the category {\sf Set}, and the above narrow context.
The general definition of a topoi is a functor from a site to a category of pre-sheaves 
(we will use this definition later, in the context of defining a notion
of forcing in cartesian categories).

An {\it elementary topoi}, is a category that enjoys closure of certain operations performed in sets.
More concretely an elementary topoi, is a category that has a subobject classifer (like the $2$ element Boolean algebra)
has finite limits and is cartesian closed.
In ${\sf Set}$ the subobject classifier is the two element set, 
and it is cartesian closed because exponentials exist. 
Also ${\sf Set}$ is closed under finite products and it has equalizers.
Topoi appeared in algebraic geometry, where they are abundant via the generalisation of a sheaf over a topological space. 
Categorially they they are a sheaf over a site (defined after the example).

Here the notion of forcing is not a Boolean algebra in the ground model, but rather a Heyting algebra. 
The Heyting algebra, will be based on a notion of forcing $P$, the site, and the elements of the topoi will be maps from 
$P^{op}$, where we reverse the order of $P$,  to ${\sf Set}$, or rather to certain pre-sheaves in ${\sf Set}.$ 
This gives rise to intuitionistic set theory. This set theory, weaker 
than than that based on the topoi ${\sf Set}$, is extremely interesting, 
for example in such a set theory, Zorns lemma is not equivalent to the axiom of choice.

\begin{example}

Consider the category ${\sf Set}^P$ of sets varying over a partially ordered set. 
Objects are functors $F:P\to Set$, that is maps $F$ which assign to each $p\in P$ a set $F(p)$ and to each $p,q\in P$ 
such that $p\leq q$ a map $F_{pq}:F(p)\to F(q)$ satisfying that for $p\leq q\leq r$ 
$F_{qr}\circ F_{pq}=F_{pr}$. An arrow $\eta:F\to G$ 
is a natural transformation between $F$ and $G$, in this case, an assignment of a map $\eta_p: F(p)\to G(p)$ 
to each $p\in P$ in such a way that $$\eta_q\circ F_{pq}=G_{pq}\circ \eta_p.$$ 
A truth value object $\Omega^P$ is determined as follows. 
A subset $U$ of $O_p=\{q\in P: p\leq q\}$ such that $q\in U$, $r\leq q$ implies
$r\in U$ is said to be upward closed over $P$. Then
$\Omega(p)$ is the family of all upward closed sets over $P$, and $\Omega_{pq}(U)=U\cap O_q$ for $p\leq q$, $U\in \Omega(p)$.
The terminal object $1$in $Set^P$ is the functor $P$ with contant value 
$\{0\}$, $true:1\to \Omega$ $true_p(0)=O_p$.

In this concrete example objects in $Set^{P^{op}}$ , where $P^{op}$ is the partially ordered set obtained by reversing the order on $P$, 
are the {\it pre-sheaves}. If $F$ is a presheaf, $x\in F(p)$ and $q\leq p$, we write $x\upharpoonright_F  q$ for $F_{pq}(x)$.

Now let $H$ be a complete Heyting algeba, associated with the partially ordered set $P$. This has universe  
$\O_p=\{q\in P: q\leq p\}$. (This makes $P$ a topological space, so the above definition applies).
and, further,  $P$ embeds into $H,$ via $p\mapsto \O_p.$ 

The presheaf on $H$ 
is a sheaf if whenever $p=\bigvee_{i\in I}p_i$ in $H$ and $s_i\in F(p_i)$ for all $i\in I$ 
satisfy $$s\upharpoonright_F(p_i\cap p_j)=s_j\upharpoonright_F(p_i\cap p_j)$$ for all $i,j\in I$, then there is a unique $s\in F(U)$ such that
$s|p_i=s_i$ for all $i\in I$. The category $Sheave(H)$ has objects as sheaves and 
as arrows, the arrows between such objects viewed as pre-sheaves. We wil show that this category is equivalent to 
to the Heyting valued model $V^L$.  

Let $L$ be any algebraic structure in a category ${\sf \C}$. 
We can form a category $Set_L$ based on $L$. The objects are pairs $(A,\alpha)$ where $\A\in {\sf C}$, 
$\alpha:A\times A\to L$ is a map such that,
$\alpha(x,y)\leq \alpha(x,x)\land \alpha(y,y)$,  $\alpha(x,y)=\alpha(y,x)$ 
and $\alpha(x,y)\lor(\alpha(y,y)\to \alpha(y,x)\leq \alpha(x,z)$. 
The morphisms between the objects $(A,\alpha)$ $(B,\beta)$ are maps $f:A\to B$ 
such that $(\forall x,y\in A)(\beta(f(x),f(y))\geq \alpha(x,y)$ and $(\forall x\in A)(\alpha(x,x)=\beta(f(x),f(x)))$. 

Now if $L$ is a Heyting algebra or, for that matter a Boolean algebra, one can form $V^L$, the 
universe of sets perturbed by $L$, the usual way.
We we can obtain another (apparently different) category as follows. First we identify elements $u,v\in V^L$ for which $||u=v||=1$.
The objects of $Set^L$ are the identified objects and arrows are those identified $f\in V^L$ such that $||f\text { is a function }||=1$.

All three categories are equivalent.

\end{example}

Passing from the concrete to the abstract, we now give a general definition, of pre-sheaves, and sheaves. 
We abstract away from topological space to sites.
Let $V$ be any category. 
A $V$ valued pre-sheaf $F$ on a category $C$ is a functor. A pre-sheaf is defined to be a Set-valued pre-sheaf, but its domain is $C^{op}$.
That is, a  pre-sheave is a functor from $C^{op}$ to ${\sf Set}$.
In our example the site was a Heyting algebra (viewed naturally as category, with arrows or morphisms reflecting the order). 

If $C$ is the poset of open sets in a topological space  interpreted as a category, then one recovers the usual notion of  
pre-sheaf on a topological space. (This will be our definition of sheaves of locally algebraised sites).

The class of all such functors, for a fixed site $C$, is denoted by $[C^{op}, Set]$ is a category
called the category of pre-sheaves, where a  morphism of pre-sheaves 
is defined to be a natural transformation  of functors. 

This makes the collection of all pre-sheaves into a category.  
The reflective subcategory consisting consisting of {\it pre-sheaves that can be glued} 
was our category ${\sf Set_L}$, in fact it was all three given in the last example. In particular one can view the comulative heirarchy 
$V_L$ as a functor category of {\it sheaves} on $L$ by reversing the order of the latter.
We formalize the notion of glueing. Our next definition tells us what will be glued.
\begin{definition} 

A site is a cartesian category $C$ with a notion of localization, i,.e for every $A\in |C|$ there are given a non empty class $Loc(A)$ 
of families of morphism $(A_i\to A)_{i\in I}$ of $C$ called the localizations of $A$ which are stable under pullbacks.
\end{definition}

A sheaf over $C$ is a functor $F: C^{op}\to Sets$ satisfying the following for every $(f_i:A_i\to A)_{i\in I}\in Loc(A)$:
\begin{enumroman}
\item if $\eta, \mu\in F(A)$ are such that $\eta_i=F(f_i)(\eta)=F(f_i)(\mu)=\mu_i$, for all $i\in I$, then $\eta=\mu$.
\item If $(\eta_i)_{i\in I}$ is a family such that  $\eta_i\in F(A_i)$ for all $i\in I$ and is compatible, i.e the diagram
$$\eta_i\in F(A_i)\to F(A_{i_A}\times A_j)\to F(A_i)\in \eta_i$$
obtained via $F$ from
$$A_i\to A_{i_A}\times A_j\to A_j$$ we have
$$F(\pi_i)(\eta_i)=F(\pi-j)(\eta_i),$$
for all $i,j\in I$, then there is $\eta\in F(A)$ such that
$$\eta_i=F(f_i)(\eta),$$
for all $i\in I$.
\end{enumroman}
We call $Sh(C)$ to be the full subcategory of sheaves of the functor category $[C^{op}, Set]$.
Now this time passing from the concrete to the abstract, we first
make the following observation.
$Set_L$ is both a topoi and a functor 
category consisting of sheaves, so this prompts the following very general definition of a topoi:

\begin{definition} A topos is a a functor category of sheaves. A Grothendieck topos os a functor category of sheaves whose domain is 
a site.
\end{definition}

We give  yet anothera narrow more concrete definition of a sheaf that suffices for our purpose. $X$ will be the prime spectrum of some subalgebra $\B$
of a  reduct of a larger algebra $\A$. The target category will be $Set.$
The functor takes a basic open set $N_a\subseteq X$, $a\in B$, 
to its characteristic function on $X$.   


Formally:

\begin{definition} Let $\A$ and $\B$ be algebras with $\B\subseteq \mathfrak{Rd}\A$.  Let $X$ be the prime spectrum of $\B$.  
A stalk at $x\in X$ is the algebra $\G_x=\A/\Ig^{\A}{x}$. 
A sheaf is a triple $(X, \delta, \pi)$, where $\delta$ is a topological space with underlying set $\bigcup \G_x$, a disjoint union of stalks,
and $\pi:\delta \to X$ is defined by
$\pi(s)= x$, where $s\in G_x$.
For $a\in A$,  let  $\sigma_a: X\to \delta$ by $\sigma_a(x)=a/\Ig^{\A}x\in \G_x$, then 
the  topology on  $\delta$ is the smallest topology  
for which all these functions are open.
\end{definition}
Here, $\delta$ glues the stalks, via a disjoint union, and the topology it gets varies them continously. 
The information coded in the stalks, lifts to a global dimension via this glueing.

We start by concrete example addressing variants and extension first order logics. 
The following discussion applies to $L_n$ (first order logic with $n$ variables), $L_{\omega,\omega}$ (usual first order logic), 
rich logics, Keislers logics with and without equality, finitray logics of infinitary
relations; the latter three logics are infinitary extensions of first order logic, 
though the former and the latter have a finitary flavour, because quantification is taken only
on finitely many variables. These logics have an extensive literature in algebraic logic.

\begin{example}

Let $L$ is a multi-dimensional modal logic, then a theory $T$ of this logic 
can be represented as the continuous sections of  a sheaf.
More precisely, a theory is determined by all complete theories containing it, that is by the Stone space, $X_T$
of $\Zd(\Fm_T)$.
One takes $\delta T$ to be the following disjoint union
$$\delta=\bigcup_{\Delta\in X_T}\{\Delta\}\times \Fm/{\Delta}.$$ 
On $\delta T$, one takes the product topology with basic open sets  
$$B_{\psi,\phi}=\{(\Delta, [\phi]_{\Delta}), \psi\in \Delta, \Delta\in X_{\Gamma}\}.$$
Then $(X_T, \delta T, \pi)$ is a {\it sheaf}, where $\pi: \delta\to X_T$ is defined for $s\in \delta$,
via 
$$\Delta\times \phi/\Delta\mapsto \Delta.$$  
Furthermore, the set $\Gamma(X_T, \delta)$ of continous maps, 
with operations defined pointwise, is isomorphic to $\Fm_T$, 
via
$$\eta(\phi_T)\mapsto \sigma_{\phi},$$ 
where $$\sigma_{\phi}(\Delta)=\phi_{\Delta}.$$ 
 
\end{example}
The glueing of the $G_x$'s amounts to taking the product  $\prod_{\Delta} \Fm/\Delta$ which is a quotient of $\Fm$ by 
$\bigcap T_i$ where each $T_i$ is a complete extension of $T$. 
So each stalk, gives some information about $T$, all together gives an exact information 
about $T$.  

\begin{example} Let $\A=\prod_{i\in I}\B_i$, where $\B_i$ are directly indecomposable $BAO$s. Then
$\Zd\A= {}^I2$ and $X(\A)$ is the Stone space of this algebra.
The stalk $\delta_{M}(\A)$ of $\A^{\delta}$ over $M\in X(\A)$ is the ultraproduct 
$\prod_{i\in I}\B_i/F$  where $F$ is the ultrafilter on $\wp(I)$  corresponding to $M$.
\end{example}

\section{Weaker structures}

Now we define certain topologies, that give rise to classes of topological spaces,
that are duals to various algebraic structures all having a distributive lattice reduct.

\begin{definition}

Let $(X,\leq )$ is a partially ordered set, and $\tau$ be a topology on $X$.
$(X,\tau)$ is called a Priestly space if

(a) $\tau$ is a Stone space, 

(b) For any $x,y\in X$ such that $x\nleq y$ there is a downward clopen set $U$ such that $y\in U$ and $x\notin U$.
(Downward here, means that when $u\in U$ and $v\leq u$, then $v\in U$.

\end{definition}

\begin{definition}
\begin{enumarab} 
\item A non-empty subset  $I$ of a partially ordered set $(P, \leq )$ is an ideal if the following conditions hold:

(a) For every $x\in I$, $y \leq x$ implies that $y\in I$ ($I$ is a lower set).

(b) For every $x,y\in L$ there is a $z\in I$ such that  $x\leq z$ and $y\leq z$ ($I$ is a directed set).
\item $I$ as above is a prime ideal if for every elements $x$ and $y$ in $P$, $x\land y\in P$ implies $x\in P$ or $y\in I.$ Here $x\land y$ denotes 
$inf\{x,y\}$; it is maximal if it is not properly contained in any proper ideal.

\item A lattice is simple if has only the universal congruence and the identity one.

\end{enumarab}
\end{definition}

\begin{definition}

Let $V$ be the class of bounded distributive lattices, and let $L\in V$.
We consider lattices as algebraic structurses $(L, \land , \lor, 0, 1)$.
Then  $Spec(L)$, the set of prime ideals,  endowed with the the topology whose base is of the 
form $N_a=\{P\in Spec(L): a\notin P\}$ and their complements
is called the Priestly space corresponding to $L$, or simply, the Priestly space of $L$.
\end{definition}

Let $Pries$ be the category of Priestly spaces, where morphisms are homeomorphisms.
We regard $V$ as a concrete category whose morphisms are algebraic homomorphisms (preserving the operations).

Let $F: V\to Pries$ be the functor that takes $L$ to its Priestly space, with the image of morphisms defined by $F(h(P))=h^{-1}(P)$. 
This is an adjoint situation, 
with its inverse the contravariant functor which assigns to 
to a Priestly space the lattice of clopen downward sets and images of morphisms are given via $\Delta(f)(U)= f^{-1}(U)$.

In algebraic logic quantifier logics, like first order logic and other variants therefore, quantifiers are treated as connectives. 
This has a modal formalism, as well, 
which views quantifies (and their duals) as boxes and diamonds, that is, as a multi -dimensional modal logic.
The algebraic framework of such muti-dimensional modal logics, or briefly multi-modal logics
when their propositional part is classical, is the notion of a Boolean algebra with operators $(BAOs)$.

We now deal with much weaker algebraic structures, namely, 
bounded distributive lattices with operators (reflecting quantifiers), denoted by $BLO$s.
This notion covers a plathora of logics starting from many valued logic, 
fuzzy logic, intuitionistic logic, multi-modal logic, 
different versions (like extensions and reducts) of first order logic.

\begin{example}

Let $\L$ be the  predicate language for $BL$ algebras, $\Fm$ denotes the set of $L$ formulas, and $\Sn$ denotes the 
set of all sentences (formulas with no free variables). 
This for example includes $MV$ algebras; that are, in turn, algebraisations of many 
valued logics.
Let $X_T$ be the Zarski (equivalently the Priestly) topology on $\Sn/T$ based on $\{\Delta\in Spec(\Sn): a\notin \Delta\}$. 
Let $\delta T=\bigcup _{\Delta\in X_T}\{\Delta\}\times \Fm_{\Delta}$. 
Then again, we have $(X_T, \delta T)$ is a {\it sheaf}, and its dual consisting of the continuous sections with operations defined pointwise, 
$\Gamma(T,\Delta)$ is actually isomorphic to $\Fm_T$. 
\end{example}

\begin{definition}

A $BLO$ is an algebra of the form $(L, f_i)_{i\in I}$ where $L$ is a distributive bounded lattice, $I$ is a set (could be infinite) 
and the $f_i$'s are unary operators that preserve order, and joins, and are idempotent $f_if_i(x)=f_i(x)$,  
on $L$, such that $f_i(0)=0$, $f_1(1)=1,$ 
and if $x\in L$, and $\Delta x=\{i\in I: f_i(x)\neq x\}$, then $\Delta( x\lor y)\subseteq \Delta x\cup \Delta y$ 
and same for meets. 
\end{definition}

\begin{definition} Let $\A=(L, f_i)_{i\in I}$ be a $BLO$. Then a subset $I$ of $\A$ is an ideal of $\A$, if $I$ is an ideal of $L$ 
and for all $i\in I$, and all $x\in L$, if $x\in I$, then $f_i(x)\in L$
\end{definition}
What distinguishes the algebraic treatment of logics corresponding to such $BLO$s, is their propositional part; 
it can be a $BL$ algebra, an $MV$ algebra, a Heyting algebra, a Boolean algebra 
and so forth.

Now our desired end, is to represent such structures 
as the continuous sections of sheaves;  the representation in this geometric will be implemented by a contravariant functor.

\section{Duality}

Let us formalize the above concrete examples in an abstract more general setting, that allows further applications.

Let $\A$ be a bounded distributive lattice with extra operations $(f_i: i\in I)$. $\Zd\A$ denotes the distributive 
bounded lattice $\Zd\A=\{x\in \A: f_ix=x,\  \forall i\in I\}$, where the operations are the natural restrictions.(Idempotency of the $f_i$s guarantees 
that this is well defined). 
If $\A$ is a locally finite  algebra of formulas of first order logic 
or predicate modal logic or intiutionistic logic, or any predicate logic where the $f_i$s are interpreted as the existential 
quantifiers, then $\Zd\A$ is the Boolean 
algebra of sentences.

Let $\K$ be class of bounded distributive lattices with extra operations $(f_i: i\in I)$.
We describe a functor that associates to each $\A\in \K$, and $J\subseteq I$, a pair of topological spaces
$(X(\A,J), \delta(\A))=\A^d$, where $\delta(\A)$ has an algebraic structure, as well; in fact it is a subdirect product
of distributive lattices.

This pair is  called the dual space of $\A$.
For $J\subseteq I$, let $\Nr_J\A=\{x\in A: f_ix=x \forall i\notin j\}$, with operations $f_i: i\in J$.
$X(\A,J)$ is the usual dual space of $\Nr_J\A$, that is, the set of all prime ideals of the lattice $\Nr_J\A$, 
this becomes a Priestly space (compact, Hausdorff and totally disconnected), when we take the collection of all sets
$N_a=\{x\in X(\A,J): a\notin x\}$, and their complements, as a base for the topology. 

For a subset $Y$ of an algebra $\A$ we let $\Co^{\A} Y$ denote the congruence relation generated by $Y$ (in the universal algebraic sense).
This is defined as the intersection of all congruence relations that have $Y$ as an equivalence class.
Now we turn to defining the second component; this is more involved. 
For $x\in X(\A,J)$, let $\G_x=\A/\Co^{\A}x$  and 
$\delta(\A)=\bigcup\{\G_x: x\in X(\A)\}.$
This is clearly a disjoint union, and hence 
it can also be looked upon as the following product $\prod_{x\in \A} \G_x$ of algebras. 
This is not semi-simple, because $x$ is only prime,
least maximal in $\Nr_J\A$, and even if it was, there is no guarantee that the congruence it generates in the big algebra is maximal. 
But, when it is, that is when $\prod_{x\in \A}\G_x$ is semi-simple case 
will deserve special attention.

The projection $\pi:\delta(\A)\to X(\A)$ is defined for $s\in \G_x$ by $\pi(s)=x$. Here $\G_x=\pi^{-1}x$ is the stalk over $x$. For $a\in A$, 
we define a function
$\sigma_a: X(\A)\to \delta(\A)$ by $\sigma_a(x)=a/\Ig^{\A}x\in \G_x$. 

Now we define the topology on 
$\delta(\A)$. It is the smallest topology  for which all these functions are open, so $\delta(\A)$ 
has both an algebraic structure and a topological one, and they are compatible.

We can turn the glass around. Having such a space we associate a bounded distributive lattice in $\K$.
Let $\pi:\G\to X$ denote the projection associated with the space $(X,\G)$, built on $\A$.
A function  $\sigma:X\to \G$ is a section of $(X,\G)$ if $\pi\circ \sigma$ is the identity 
on $X$. 

Dually, the inverse construction  uses the sectional functor.
The set $\Gamma(X,\G)$ of all continuous sections of 
$(X,\G)$ becomes a $BLO$ by defining the operations pointwise, recall that $\G=\prod \G_x$ is a product of bounded distributive lattices.

The mapping $\eta:\A\to \Gamma(X(\A,J), \delta(\A))$ defined by $\eta(a)=\sigma_a$ 
is as easily checked  an isomorphism. 

To complete the definition of the contravariant functor we need to define the dual of morphisms. 
These are  natural transformations corresponding to the defining functors of the sheaves, but more concretely:

Given two spaces $(Y,\G)$ and $(X,\L)$ a sheaf morphism $H:(Y,\G)\to (X,\L)$ is a pair $(\lambda,\mu)$ where $\lambda:Y\to X$ is a continous map
and $\mu$ is a continous map $Y+_{\lambda} \L\to \G$ such that $\mu_y=\mu(y,-)$ is a homomorphism of $\L_{\lambda(y)}$ into $\G_y$.
We consider $Y+_{\lambda} \L=\{(y,t)\in Y\times \L:\lambda(y)=\pi(t)\}$ as a subspace of $Y\times \L$.
That is, it inherits its topology from the product topology on $Y\times \L$.

A sheaf morphism $(\lambda,\mu)=H:(Y,\G)\to (X,\L)$ produces a homomorphism of lattices
$\Gamma(H):\Gamma(X,\L)\to \Gamma(Y,\G)$ the natural way:
for $\sigma\in \Gamma(X,\L)$ define $\Gamma(H)\sigma$ by $(\Gamma(H)\sigma)(y)=\mu(y, \sigma(\lambda y))$ for all $y\in Y$.
A sheaf morphism $h^d:\B^d\to \A^d$ can also be asociated with a homomorphism $h:\A\to \B$. 
Define $h^d=(h^*, h^o)$ where for $y\in X(\B)$, $h^*(y)=h^{-1}\cap Zd\A$ and for $y\in X(\B)$ and $a\in A$
$$h^0(h, a/\Ig^{\A}h^*(y))=h(a)/\Ig^{\B}y.$$

This is indeed a generalization of the Stone duality.
Indeed, given a Boolean algebra $\A$, then in this case the stalks are just the $2$ element Boolean algebra, 
and the dual space is just $Hom_{Bool}(X,2)$ the Stone space (here the topology is the subspace topology of the Cantor set), 
and  the double dual,  are the continuous functions from the Stone space  to the two element discrete space, with operations 
defined pointwise, namely 
$Hom_{top}(Hom_{bool}(X,2),2),$ where the last is the  $2$ element discrete space
and the first is the $2$ element Boolean algebra. 
Both are co-separators in their category,
and the existence of a coseparator $C$ defines a natural 
isomorphism via the contravariant Hom functor $Hom(-,C)$, applied twice.This is naturally isomorphic to $A$.
(Other similar contexts are the duality between $C$ star algebras and Compact Hausdorf spaces (via the Gelfund Hom functor; here 
$\mathbb{C}$ is the co separator), 
and the category of  abelian groups and locally  compact abelian groups, via the Pontreyagen 
Hom functor; here $\mathbb{R}/{\mathbb{Z}}$ is the co 
separator.)
We can also formulate the above natural isomorphism, in terms of $Hom$ functors, namey.

For a variety $V$ of $BLO$'s we denote by ${\mathfrak rd}V$, the class ob obtained by discarding the extra operators.
\begin{theorem} Let $V$ be an algebraic category whose objects have a distributive lattice structure. Assume that there exists $C\in V$,
such that the contravariant hom functor $Hom(-, C)$ is an equivalence between ${\mathfrak Rd}V$ and $Priest({\mathfrak Rd}V)$.
Let $\A\in V$, let $\Zd\A$ be its zero-dimensional part, and $X$ be the Priestly space of $Zd\A$. Let 
$\delta$ be the disjoint union of stalks with the topology as defined above, that is $\delta=\bigcup_{x\in X} \G_{x}$, with $\G_x=\A/\Ig^{\A}(x)$,
with topology  induced by the maps $\sigma_a: X\to \delta$ via $x\mapsto a/\Ig^{\A}x$, for every $a\in \Zd\A$.
Then $$\A\cong Hom_{top}(Hom(\Zd(\A), 2), \bigcup_{x\in X} \G_x),$$ and the isomorphism is natural.
\end{theorem}
\begin{definition} Let $\A\in \CA_{\omega}$ and $x\in A$. The dimension set of $x$, in symbols $\Delta x$, 
is the set $\{i\in \omega: c_ix\neq x\}.$ Let $n\in \omega$. 
Then the $n$ neat reduct of $\A$ is the cylindric algebra of dimension $n$ 
consisting only of $n$ dimensional elements (those elements such that $\Delta x\subseteq n)$,
 and with operations indexed up to $n$.
\end{definition}

\begin{example}

\begin{enumarab}

\item Let $\A\in \Nr_n\CA_{\omega}$. Then there is a sheaf ${\bf X} =(X, \delta, \pi)$ such that $\A$ is isomorphic to continuous sections
$\Gamma(X;\delta)$ of $\bold X.$  Indeed, let $X(\A)$ be the Stone space of $\Zd\A$. Then for any maximal ideal 
$x$ in $\Zd\A$, $\Ig^{\A}(x)$ is maximal in $\Nr_n\A$.
Let $\delta(A)=\bigcup \G_x$, where $\G_x=\A/\Ig^{\A}x$. 
The projection $\pi:\delta(\A)\to X(\A)$ is defined 
for $s\in \G_x$ by $\pi(s)=x$. For $a\in A$, 
we define a function $\sigma_a: X(\A)\to \delta(\A)$ by $\sigma_a(x)=a/\Ig^{\A}x\in \G_x$. 
Then $\pi\circ  \sigma$ is the identity and $\delta(\A)$ has  the smallest topology such that these maps are continuous.
Then $\eta: \A\to \Gamma(X(\A)), \delta)$ defined by
$\eta(a)=\sigma_a$ is the desired isomorphism.

\begin{example}

Let $\Sn_{\L_n}$ denote the set of all $\L_n$ sentences, and  fix an enumeration $(c_i: i<n)$ of the constant symbols.
We assume that $T\subseteq  \Sn_{\L_0}$. 
Let $X_T=\{\Delta\subseteq \Sn_{\L_0}: \Delta \text{ is complete }\}$.
This is simply the underlying set of the Priestly  space, equivalently  the Stone space, 
of the Boolean algebra $\Sn_{L_0}/T$. For each $\Delta\in X_T,$ let $\Sn_{\L_n}/{\Delta}$
be the corresponding Tarski-Lindenbaum quotient algebra, which is a (representable) cylindric algebra of dimension $n$. 
The $i$th cylindrifier $c_i$ is defined by 
$c_i\phi/{\Delta}=\exists \phi(c_i|x)$, where the latter is the formula
obtained by replacing the $i$th constant if present by the first variable $x$ not occurring in $\phi$, 
and then applying the existential quantifier $\exists x$.
Let $\delta T$ be the following disjoint union
$\bigcup_{\Delta\in X_T}\{\Delta\}\times Sn_{\L_n}/{\Delta}.$ 
Define the following topologies, on $X_T$ and $\delta T$, respectively.
On $X_{T}$ the Priestly (Stone) topology, and on $\delta_{\Gamma}$ 
the topology with base $B_{\psi,\phi}=\{\Delta, [\phi]_{\Delta}, \psi\in \Delta, \Delta\in \Delta_{\Gamma}\}.$
Then $(X_T, \delta T)$ is a {\it sheaf}, and its dual consisting of the continuous sections, 
$\Gamma(T,\Delta)$, with operations defined pointwise, is actually isomorphic to $\Sn_{\L_n}/T$. 

\end{example}

\item Let $\A\in \Nr_n\CA_{\omega}$. For any ultrafilters 
$\mu$ and $\Gamma$ in $\Zd\A$, the map $\lambda:\A/\mu\to \A/\Gamma$ defined via, 
$a/\mu\mapsto a/\Gamma$
maps $\Zd\A$ into $\Zd\A$. (The latter is the set of zero-dimensional elements). 
The dual morphism is $\lambda^d=(\lambda, \lambda^0) :(X_{\Gamma}, \delta(\Gamma))\to (X_{\mu}, \delta(\mu))$, 
is defined by $\lambda(\Delta)=\Delta$ and $\lambda^0(\Delta, (\Delta), a/{\Delta}))=(\Delta, a/\Delta)$.
Thus it is an isompphism from $(X_{\Gamma}, \delta(\Gamma)$ onto the restriction of 
$(X_{\mu}, \delta(\mu))$ to the closed set $ X_{\Gamma}$.
Conversley, every restriction of $(X_{\mu}, \delta(\mu))$ to a closed subset $Y$ of $X_{\mu}$ is up to isomorphism the dual space of 
$\Nr_n\A/F$ for a filter $F$ of $\Zd\A$. For if $\Gamma=\bigcap Y$, then $Y=X_{\Gamma}$ 
since $Y$ is closed and the dual space of $\Nr_n\A/{\Gamma}$ is isomorphic to 
$(Y, \delta(\mu)\upharpoonright Y)$.
\end{enumarab}
\end{example}


\section{Forcing in topoi}

{ Here we generalize results from topoi of sheaves to quasi topoi of pre shaeves, in the context of forcing
and many sorted theories of the higher order partial element logic.}

Topoi first appeared in algebraic geometry, where they are abundant via the generalisation of a sheaf over a topological space. 
Categorially they they are a sheaves over a site.  Let ${\sf C}$ be a category whose members have a 
distributive lattice reduct.
that ${\sf C}$ is an algebraic 
category. Let $H$ be an algebra having a lattice reduct. The one can form the universe of sets based on $H$ and $V$ the usual way, by taking at step $\alpha$ 
the set of all functions for $V_{\beta}^H$ to $2$. The question is can we capture such a notion of generic extensions using 
a forcing relation between elements in $H$ and formulas in the logic, 
so that $p$ forces $\phi$ iff for every generic filter $G$ of $H$, with $p\in G,$
we have $V[G]\models \phi^{G}$ where the $G$ interpretation of an $H$ term is defined the usual way.  
That is is $a^G=\{y: \exists p\in P:  (p,y)\in a\}$.

The Heyting  example we did forcing with respect to the topology have open sets $\O_p$, 
where $p\in P$. The set $\{\O_ p\in P\}$ forms a Heyting algebra, and is 
an instance of a site.

Now we show that our functor on a class $\K$ can be seen as a topoi, more precisely as functor from a site into 
a category of sheaves. In this case, we can naturally define a notion of forcing.

The connection between logic and topos, or rather elementary topos, is somewhat deep. 
Topos is a categarial reflection of Set, as much as abelian categories are catgorial reflection of abelian groups.
So different forms of forcing, like Cohen's, Robinson's and Kripke, can be seen as {\it forcing with sheaves on different sites.} 
Another rich source is the interpretation of higher order languages, via Henkin semantics as many sorted languages, 
viewed as type theory or logics of partial elements in topos.
Here we have a {\it completeness theorem}, indicating that topoi are {\it just the right abstraction}. 
We will discuss both related subjects via topoi, or categories of pre-sheaves.

The connections of topos to sheaves is well known. 
What concerns us here is logical topoi, which are strongly related to forcing in set theory; 
and here the connection between sheaves and topos is explict.
Dana Scott envisaged that forcing can be based on intuitionistic logic, 
by noting that the forcing consition for negation is actually an intiutinistic forcing.
Intuitionistic logic correspond to Heyting algebras, and we have already defined the comulative heirarchy of sets
corresponding to pre-sheaves on a Heyting 
algebra.

Our next theorem which is an abstraction of forcing using topoi instead of sets, relies heaviy on some categorial concepts.
We will show that given a small category that happens to be a site, then one can look at this category as
forcing conditions, and the  ramified language as its Yoneda image. 

Indeed, a locally small category $C$ (where $Hom(A,-)$ are sets, embeds fully and faithfully 
into the category of set-valued pre-sheaves via the Yoneda embdding. The presheaf 
category is (up to equivalence of categories) the free co-limit  completion of the category $C$. 

This is indeed a generalization of Boolean and Heyting forcing. For the latter we started with a Heyting algebra $H$, 
$Set_H$, the perturbed universe of sets by formulas taking value in $H$, {\it is} 
its the reflective subcategory of its {\it Yoneda image}, namely, it is the topos of sheaves on $H$.

The Yoneda lemma says that instead of studying the locally small (that is hom-sets are sets) $C$, one should study the 
category of all functors of ${\sf C}$ into Set, that is the presheaves on ${\sf C}.$ 
Each object $A$ of $C$ gives rise a hom functor  $h^A=Hom(A,-)$. 
For an arbitray functor $F$ from $C$ to set, for each object $A$ of $C,$ the
natural transformations from $h^A$ to $F$ are in one to one correspondence with the elements of $F(A)$, 
that is $Nat(h^A, F)\cong F(A)$, and contravariantly 
$Nat(h^A,G)\cong G(A)$.  Furthermore, this isomorphism is natural in $A$ and $F$, when both sides are regarded as functors 
from $^Cset\times C$ to $Set$.

\begin{theorem} Assume that $C$ is an algebraic site. 
Then our $\bold F$ can be used to define a notion of forcing; 
the forcing conditions will be members of the category.
\end{theorem} 
\begin{proof}
(In our first example $C$ was only a partially ordered set). Let $\bar{C}$ be the set of all functors from $C^{op}$ to $Set$. 
Write $y:C\to \bar{C}$ to be the Yoneda embedding. The commulative heirarchy of sets will be formed in $\bar{C}$, elements of the site
$C$ will be the forcing conditions, and a $C$-term will be represented by a functor from $C$ to ${\sf Set}$, that is an element
in $\bar{C}$. So intuitively $C$ is the forcing notion 
and $\bar{C}$ is the ramified language. 

$0$ is the initial object in $C$, and we have the operation $P(A)$ on objects abstracting the operation of power set.
The maps $i_{\alpha,\beta}:P^{\alpha}(0)\to \wp^{\beta}(0)$, for all ordinals $\alpha, \beta$ are all inclusions, 
therefore we may construct $colimit_{\alpha\in On}P^{\alpha}(0)$ as the union 
$$V^{(\bar{C})}=\bigcup_{\alpha\in On}V_{\alpha}^{\bar{C}}.$$
Using the fact that in $\bar{C}$ power objects are constructed for $A\in \bar{C}$ as 
$$P(A)(I)=Sub_{\bar{C}}(y(I)\times A),$$ 
where $Sub$ denotes subobjects and $I\in C,$
one may  consider $V^{\bar{C}}$ as defined inductively by the rules
$$a\in V^{\bar{C}}(I)\text { iff $a$ is a set valued sub-presheaf (sub-functor) of } y(I)\times V^{\bar{C}}$$
for objects $I$ of $C$, this step corresponds to the inductive step in usual set theory, when we obtain 
$V_{\alpha+1}$ from $V_{\alpha}$. 

If $a\in V^{\bar{C}}(I)$ and $u:I\to J$ is a morphism in $C,$ then we write
$a.u$ for the subobject of $y(J)\times V^{(\bar C)}$ where $(v,c)\in a.u$ iff $(uv, c)\in a$. 
This allows one to interpret  the membership relation as:
$$\epsilon\mapsto V^{\bar {C}}\times P(V^{(\bar{C}}).$$ 
Equality is defined as usual in presheaves. 
The Kripke-Joyal semantics of this interpretation in pre-sheaves over $\bar{C}$ give rise to 
the following forcing clauses, where $I\in C$, and $a,b\in V^{\bar{C}}$; 
and the language is the usual first order language for set theory:

$$I\models a\in b \Longleftrightarrow (Id_I,a)\in b$$
$$I\models a=b \Longleftrightarrow a=b$$
$$I\models \phi\land \psi (c)\Longleftrightarrow  I\models \phi(c) \text { and } I\models \psi(c)$$
$$I\models \phi\to \psi(c)  \text { iff for all $u:J\to I$ from $J\models \phi(c.u)$ it follows that $J\models \psi(c.u)$}$$
$$I\models \phi\lor \psi(c) \Longleftrightarrow J\models \phi(c)\text { or } I\models \psi(c)$$
$$I\models \forall x\phi(x,c)\Longleftrightarrow J\models \phi(a, c.u), \forall u:J\to I, a\in V^{\bar{C}}(J)$$
$$I\models \exists x\phi(x.c)\Longleftrightarrow I\models \phi(a,c).$$
Note that when $C$ is just a Heyting algebra then ${\cal C}$ is the set of all presheaves, that is functions from $C$ to $Set$.
And the above construction gives exactly intuitionistic forcing.
This is an interpretation of set theory in the 
sheaf topos $E=Sh(\bar{C}, J)$ where $J$ is a Grothendieck topology on small category $C$. 

If $C$ is not small, one can also interpret set theory in the same Sheaf topoi, with some modifications
The main obstacle here is that colimits of transfinite chains of inclusions are not simply union but 
rather union followed by sheafification $a: \bar{C}\to Sh(C, J)$ that is a left 
adjoint to $i: Sh(\C,J)\to \bar{C}$ and that the reflection maps $\eta_X:X\to a(X)$ in general cannot be 
viewed as inclusions. This can be overcome by obtaining the model for $E$ as a quotient of the model in $\bar{C}$. 
Construct by transfinite recursion a family of morphisms $e_{\alpha}: V_{\alpha}^{C)}\to V_{\alpha}^{E}$ where
$V^{C}$ and $V^{E}$ refer to cumulative heirachies in the Gothendieck toposes $\bar{C}$ and $E$. These $(e_{\alpha})$ 
satisfy $$i_{\alpha,\beta}\circ e_{\alpha}=e_{\beta}\circ i_{\alpha,\beta}$$
and they are all dense w.r.t to the topology and for 
successors they are epics in $\C$. 
Then $e:V^{C}\to V^{E}$ the unique 
mediating arrow between them is an epimorphism.
\end{proof}

\subsection{Fuzzy Forcing}

In recent years, several algebras have appeared in the framework 
of fuzzy logics; some can be readily seen as extensions of intuitionistic logic. All those algebras are based on the notion of residuated lattices.
Examples include $MV$ algebras, $BL$ algebras. 
What distinguishes, these algebras is the presence of two distinct conjunctions, $\land $ and $\otimes$, the latter
the Lukasiewicz conjunction. 

There has been work in interpreting many sorted theories in quasi-topos based on $MV$ algebras, 
but the best that was achieved, to the best of our knowlege, 
is that the interpretation is only faithfully represented in the Heyting algebra obtained from the $MV$ by identifying the two conjunctions.

For the Heyting case, a completeness theorem is obtained for such theories
using the topoi $Set_H$, where $H$ is a sub-object classifier, in the category $E(T)$ of definable terms and total functions in the many sorted 
language. This category which happens to be a topos, is a reflective subactegory of the category of fuzzy sets corresponding to $T$ .

Also, forcing in Heyting valued models is known, it is natural to ask about forcing in 
such new algebras. 
The crucial missing link here, is the interpretation of both conjunctions, particulary, the Lukasiewicz conjunction. In the latter case, 
it is not obvious how to 
define $p$ forces $\phi\otimes\psi$, and in the second case, it is not clear how to interpret the truth values of the $MV$ algebra.

First thing to notice is that the first conjunction has to do with the lattice structure, but the second lends itself to a monoidal structure.
In fact, the class of $MV$ algebras form {\it an abelian category}, which a generalization of {\sf Ab}  and $BL$ 
algebras form a {\it monoidal} category an extension of {\sf Monoid}

There is a topos theory for such algebras, where one can form $Set_L$ similar
to Heyting algebars, in fact one can construct all three equivalent categories, that are the qusi topoi of the comulative heriarchy.

But this process has to do with the first part of forcing constructions namely, perturbing the the universe of sets
via an algebra. The $V_{\beta}$s are just functions from $V_{\alpha}$ to the algebra, and the interaction of functions is defined pontwise.
In translating  the forcing in the ground model, to capture all generic extensions, 
one has to define clauses that involve the non-classical conjunction. 
In usual forcing the inductive clauses are defined by the 
'meta meaning' of the connectives, namely in English,  they are not defined intrinsically; this is the case for both classical forcing and intuitionistic 
forcing. However, what can guide us here is the duality between the bi- functors tensor products and Hom. On the very 
basic level, namely in the algebra, this is expresed by
$a\otimes b\leq c$ iff $a\leq (b\to c).$

It seems that the best way to approach this problem is to work on the abstract level.
Let us start with a $BL$ algebras. Let $\A$ be such an algebra, so besides the lattice structure it has a monodial nature.
So it is natural to declare, that a presheaf $E$ on $L$ assigns to each element $x\in L$ a 
monoid $E(x)$ and whenever $x\leq y$ assigns a morphism 
$E_{xy}:E(y)\to E(x)$, that preserves the monoid structure, meaning that it also preserves $\otimes$. 
Now $E(x,y)$ is an abstraction of equality, it defines the $L$ sets (as in the case of Heyting algebras)
and now  sheaves can be defined as equalizers, this means that the sheaves reflect the real equality, and not the fuzzy one.
Note that in $L$, $\implies $ and $\otimes$ are adjoint. 
We may even go further.

One starts with a site $C$ that has a monoid structure, and one takes $\bar{C}$ to be the category of all functors from
${C^{op}}$ to {\sf Monoid}; this will be the pre-sheaf topos. Let $y:C\to \bar{C}$ be the Yoneda embedding 
and proceed as above, except that one has to define  $P(A)$. 
This is defined by taking those fuzzy subsets of $A$ that admit pullbacks with $A^*$.
A forcing condition, where the new conjunction $\otimes$ that connects
formulas, and this will have to be induced by the monodial structure.

In fact, if you have a monoidal category of sheaves from a Grothendieck site, or simply a site,
into pre-sheaves, then you have a tensor product $\otimes$ 
in both $C$ and $\bar{C}$, the second tensor product is induced by the first in a natural way via the Yoneda Lemma,
and this defines the {\it new } conjunction $\otimes$, as follows 
$$I\models \phi\otimes \psi \text { iff there exists $J$ a monomorphism $u:J\to I$  such that } I\otimes J\models \phi\land \psi.$$
Here $I\otimes J$ is the tensor product of $I$ and $J$.
wWe will give an equivalent definition below.

Finally, Grothendieck topos are sheaves over sites, Girard completely classifies them, 
different sites can give the same topos, these are called Morita equivalent.
This can also happen in the classical case, 
when non isomorphic Boolean algebras, give the same notion of forcing, namely, they give the same extension of the ground 
model.

\subsection{Higher order logics of partial elements, and topoi}

A very much related topic is the logic of partial elements formulated 
as a many sorted theory with higher types. In this context, topoi summarizes in categorial form the essence of higher
order logic.  A complete set of axioms, with respect to semantics represented by topoi, 
can be given in a higher order language. 

We use the logic of partial elements, which is a higher order logic, as worked out in [Handbook]. We introduce an existence predicate $E$, reading 
$E\tau$ as $\tau$ potentially exists.  We intorduce the notion of equality 
of partial elements which presupposes existence in the following sense:
$$\tau=\sigma\to E\tau\land E\sigma.$$ 
We present the logic as many sorted theory with higher types. 
Definition 3.1 is the same, and we take the axioms the same as those on p.1061, replacing the intuitinistic part 
by fuzzy propositional logic. We follow the notation adopted therein.

\begin{definition}A type $A$ is a term of the syntactic form $Iy:[A]. A(\phi\leftrightarrow y(x))$, which we abbreviate by $\{x:A|\phi\}$.
For $A=\{x:A|\phi\}$ we use the notation $\tau\in A$, $\forall x\in A$ $\exists x\in A$ with ther obvious meanings.
\item A definable relation $F:A\to B$ is a closed term of the form:
$$Iz:[A,B].\forall x"A: \forall y:B(z(x,y)\leftrightarrow \phi).$$
\end{definition}
We write $$F'(\tau)\text { for }Iy:B.F(\tau,y).$$
and we define: 
$$\lambda x:A.\sigma \text { for } Iz:[A,B]. \forall x, y[z(x,y)\leftrightarrow y=\sigma].$$

The category $E(T)$ of definable types and definable total functions of $T$ has the definable types of $L(T)$ as objects and 
as morphisms from $A$ to $B$ equivalence classes of definable relations from $A$ to $B$ such that
$$T\vdash \forall x\in A.F'(x)\in B$$ where $F$ and $G$ are equivalent iff
$$T\vdash \forall x\in A.F'(x)\equiv G'(x)$$
and composition  defined by
$$F\circ G=\lambda x. F'G'(x)),$$
(here $\lambda$ is caled $\lambda$ abstraction and it is similar to the $\lambda$ operator in $\Lambda$ calculas)
is a topos and it has sub-object classifier $\Omega$, where $Hom(1, \Omega)$, the set of definable subsets
of $\Omega$ is just the Tarski-Lindenbaum algebra of $T$. 
In analogy to algebraic logic, there is an interplay between the algebraic properties of the power object
$P(A)$ and $Hom(A, \Omega)$, in the category $E(T)$ and logical properties $\Omega$. 

This higher order logic can be {\it interpreted  in any topoi}, and this gives a complete semantics. 
Furthermore, for a higher order many sorted theory $T$ it is 
enough to interpret it in $E(T)$, which is categorially equivalent to $Set_{\Omega}$.  
This is done by given truth values to $\Omega$, making it a Heyting algebra, in analogy
to defining Heyting valued models for set theory. So here a notion of {\it forcing} is implicit, and that is the forcing in the ground model, 
namely, for $p\in \Omega$, $p$ forces $\phi(x)$ iff $M[G]\models \phi(x^G)$, for every generic filter $G$ of $\Omega$.
Here ${x}^G$ is the $G$ interpretation of the $\Omega$-term $x$. 

Recall that $Set_{\Omega}$,  can be viewed as a category of set valued 
pre-sheaves, but  when $\Omega=\O(X)$, the algebra of open sets of a topological space, 
 then $Set_{\Omega}$  is equivalent to only set valued sheaves over the site $X$.
On the other hand, $Set_{\Omega}$ based on the Heyting algebra $\Omega$, establishes 
the tie between type theory, Lambda Calculas, and set theory.

A similar task can be done using $MV$ algebras, giving the category of fuzzy sets. 
Indeed let $\A$ be an $MV$ algebra. In this case,  one can show that the category of definable terms and total functions,
call it  $CAT$, which is a category of fuzzy sets,  can be defined. 

Consider the category of $H$ sets $CAT_H$, $H$ a Heyting algebra  and the category of fuzzy sets $CAT_A$ based on an $MV$ algebra $A$.
We can asume that $CAT_{Hey(A)}\subseteq CAT_A$, by identifying $\otimes$ and $\land$ in the $MV$ algebra on which 
$CAT_H$ has  has a power object $PA$. It is $P1=(Sub(To), )=H$. 

Another category called complete $H$ category can be defined, this is equivalent to $CAT_H$.
If $A$ is an $MV$ algebra, then $Cat_{Hey(A)}$ is refletive subcategory of $CAT_A$.



Consider the category $CAt_H$. For any $H$-set $\A$ there exists a power set of $A$. $P(A)=(Sub(A), \rho)$ where $s\rho t$ iff 
$$\bigvee_{x\in A}(s(x)\leftrightarrow t(x))$$
Then it can be shown that there exists a natural isomorphism $\phi:Hom(-, PA)\to Sub(-,\times A)$. 

The suboject classifier is readily defined by $P1=H$. A singleton is a 
function $s:A\to H$ such that $\forall x,y\in A)(s(x)\land \delta(x,y)\leq s(y)$ and $\forall x, y\in A) (s(x)\land s(y)\leq \delta(x,y)$. 
An $H$ set is complete if for every singleton $s$ there is a unique $a\in A$ such that $s=\{a\}$. 
For any $H$ set there exists an isomorphic complete $H$ set.

Now given $CAT$ based on the $MV$ algebra $A$. Identify $\otimes$ with $\land$, getting a Heyting algebra $H$. 
Take the subategory consisting of complete $H$ sets, with strong morphisms defined as they are in $CAT$. 

Then if $f:\A\to B$ is a morphism in $CAT$ between complete $H$ sets, then its graph $f:\A\times B\to \Omega$ is defined by $[f](x,y)=\rho(f(x,y)$.
Then let $A, B$ be complete sets and let $f:\A\to B$ be a morphism. Then there exists a strong morphism $f^*:A\to B$ such $[f^*]=f$.
$CAT_H$ is equivalent to the sub-category consisting of {\it complete sets}, which is a reflective subcategory of $CAT$. 


Now let us go back to interpretations of many sorted higher order theoreies. 
Dropping the Lukasiewicz connective
$\otimes$, the resuting reduct, call it also $E(T),$ is naturally isomorphic to $Set_H$, 
where $H$ is the Heyting algebra obtained by identifying the Luckasiewick conjunction
with the usual one in $\A$. Furthermore, $E(T)$ is a reflective sub-category of  $CAT$, 
where $CAT$, a category of fuzzy sets (defined by modifiying slightly the definition of 
$Set_L$ for Heyting algebras) is only a quasi-topos.

$CAT$ is a category all the same, it is complete and it has pullbacks, but is not a topos.
In particular, it does not have a sub-object classifier.  

Furthermore, for any sort $A$, one can define $\bar{A}$ as follows:
$$\{x: [A]|\forall y,z:A((x(y)\land x(z))\to (y=z\land y\in A))\}$$
with embedding
$$\mu_A=\lambda x.\{z: A|x=z\}.$$ 
like so that $A$ is a subobject of $\bar{A}$ via the morphism
$\mu_A$.

Now for every object $(Sub(A), \leq)$ is a residuated lattice,  we do not have a uniform subobject classifier, 
though we have an almost all sub-object classifier, of the reflective subcategory $Set_H$. 

Then we can interpret the many sorted fuzzy theories into $CAT$ defining the values of $[\tau]$ 
and $[\phi]$ as products for every term and formula. 
We define $[E\tau]$ as $\mu\circ \tau.$ 

So for every object in the big category embeds into a unique object in the small category, obtained by identifying $\otimes$ with $\land$. 
Now how can we extend the truth values of $H$ to the bigger category, namely to $\A$? 
$H$ defines an implicit notion of intuitionistic forcing and $\A$ defines a quasi-topoi of $\A$ sets.
So here again we are forcing with {\it pre-sheaves on different sites}, namely $H$ and $\A$. 
We have already defined the forcing relation for $p\in A$, and the Lukasiewicz connective $\phi\otimes\psi$, 
all we have to do is lift it to $Set_{\A}$. Another way is, to lowe it to $\A$, namely:

\begin{theorem} Let $A$ be a complete  $MV$ algebra. Let $A^*=A\times A$. 
Then one defines $P(B)$ as the set of all subsets of $B$, that gives pullbacks in the category of fuzzy sets.
Then one can define truth values of $\times$ and $\land$, in $\A^*$, completing the interpretation.
\end{theorem}


Now let us pause for a minute and take our breath to reflect on what we have already did.

Let $L$ be an $MV$ algebra. Let $H$ be its Heyting reduct obatined by deleting the Lukasiewicz conjunction. 
We can form $V^L$ the comulative heirarcy of sets,
which is an $L$ valued model of set theory, and similarly we can prove $V^H$. 

$V_L$ is the same as $Set_L$, which is a functor category,
a category of pre-seheaves, and only a quasi-topoi. We can also form $V_H$ and $Set_H$, which is also a functor category,
a category of sheaves and a topoi. The latter is a reflective sub-category of the former.

Given a many sorted theory higher order  $T$, one can form the topoi, $E(T)$, for which $Set_H$ forms a complete semantics, where
$H$ is the subobject classifier in the topoi, endowed with truth values corresponding to the connectives. We also have a notion of $H$
forcing for intuitionistic set theory, and this forcing relation captures genericity, meaning that for $p\in H$, if $p$ forces $\phi$, then $\phi$ is true
in any generic extension of $V^H$ obtained by a filter $G$ of $H$.
  
Now we {\it did} define a notion of forcing for many valued set theory, 
that does exactly the same thing, namely, the forcing relation, appropriately defined at the Lukasiewicz conjunctions, 
captures genericty.  Starting for many sorted logic of partial elements, one can define the category $CAT$ corresponding to $E(T)$.
$CAT$ is not a topoi, and it does not have a sub-object classifer, but it has one, that is almost one.
Then one can define an interpretation of such theories in $Set_L$, the same way, but the problem is how to define the semantics of the Lukasiewicz 
conjunction. This can be done like defining the notion of forcing.

So all our results, extending results from topoi of sheaves, to quasi-topoi of pre sheaves, 
were actully inspired by the following:

Let $C$ be a site, and let $Sh(C)$ be the functor subcategory of $[C^{op}, Set].$
Then the Yoneda lemma tells us, basically, that we have:
$$Hom(C,E)\to Hom_{Top}(E, Sh(C)).$$

Expressed in words the topos of all sheaves on a cite, is a reflective subcategory of the quasi-topos-
sheaves on this site.  

When we force with Heyting algebras, then we are capturing the global generic extension locally by forcing conditions, the global extension 
$V^H$, is in fact a category of sheaves. 

When we deal with higher order logic, then we end up with a topos with a Heyting algebra as a subobject-classifier.
This provides  a complete semantics 
for the theory.

Now intuitionistic set theory can be formulated as a a many sorted logic in higher order logic, 
and in this case,  this topos is no more than $V^H$,
where $H$ is Heyting algebra.

So in both cases we end up with a comulative heirachy of sets. Now we have two semantics on $V^H$. But this is only apparent.
The semantics obtained by forcing is  {\it the complete} semantics for set theory, obtained by the interpretation of set theory into topoi

Therefore, if one defines forcing corresponding to the Lukasiewicz connective, then this complete the definition of the interpretation
of many valued theories into a quasi topos.

\begin{theorem} 
\begin{enumarab}
\item Let $A$ be an $MV$ algebra and let $A^*$ be as defined. 
Let $P(A)$ be the set of al fuzzy subsets of $A$ that has a pullback. Then there is an interpaly 
between $P(A)$ and $Hom(-, A^*)$, in the sense that $Set_H$ is the category of fuzzy sets, it is a quasi-topoi, and a category of presheaves, that is
functors from $A$ to ${\sf Set}$. 
\item   Let $E(T)$ be as defined above. Then $E(T)\cong Set_H$, where $H$ is an $MV$ algebra, with truth values as defined above 
The algebra $H$ can be turned into a forcing notion, in the ramified language  $Set_H$, which is  the category of presheaves, 
i.e functors from $E$ to $Set$, applying the Yoneda lemma. To fine the forcing clauses, 
one transfers the truth values of of $\otimes$ and $\land$, defined on $A^*$,  to $A$, via the morphim from 
$\A\to  \A^*$.
\end{enumarab}
\end{theorem}
Wrapping up, what we have been  doing so far:

\begin{enumarab}

\item We extend results of topoi of sheaves, which forms a reflective subcategory of the quasi-topos of  pre-sheaves. 

\item In this context, we have been forcing with {\it two sites}, the first forcing is 
ordinary forcing in the sense that the meta language is usual first order thery.
This was implemented using the Yoneda lemma in category theory, which says that you can look at a site, in the more concrete case 
a Heyting algebra as a category of functors on 
this site, to {\sf Set}.This is a natural isomorphism (in the categorial sense, in more than one parameter). 

This is also a natural generalization of ordinary forcing 
where the site represents the forcing conditions and the set of all such functors represents the ramified
language. The advantage to work on the level of sheaves and topoi, 
is that this view lends itself to other generalization, like Kripke forcing, and intuitionistic forcing.
This also enabled us to view the comulative heirarchy of sets $V^H$, based on a Heyting algebra,
as a topos of sheaves, and this is precisely 
the image of the Yoneda lemma applied to the forcing conditions, namely, the site, which is $H$.

From the different view of interpreting many sorted
higher order logic in topoi, we ended up at the same point. If we restrict to set theory (formulated as a higher order
logic), then the completeness theorem, is provided by
giving truth values to the sub-object classifier of the resulting topos of definable sets and total functions.
This is the same topos, as above, namely $V^H$. 

\item Now, the forcing done above was implemented by forcing (on sheaves) on different cites, but  {\it on sheaves}.
It is very natural therefore, to ask whether the forcing can be implemented by {\it pre-sheaves}, which is a larger functor category.
Many valued logic gives this forcing. Indeed, in this case, we have a new connective, namely Lukasiewicz conjunction, 
so we are actually working outside first order logic. Furthermore, the universe of sets $Set_L$ 
defined above for an $MV$ algebra $L$ consists of all presheaves on $L$, in particular, it is only
a quasi-topoi.
In usual forcing, even in topoi, other complex clauses are defined from the primitive ones using the usual meaning of $\land$ $\lor$,etc.
But when we have a fuzzy connective, 
we have to think diferently. 

So we observed two things. An $MV$ algebra has a monodial facet. Furthermore, in abelian categories,
we have the bifunctors tensor and Hom are dual, expressed on the algebra level we have $a\otimes b\leq c$ iff $a\leq b\to c$.
In usual topoi forcing, the Yoneda lemma, takes us to sheaves, that is functors from the cite to ${\sf Set}$. The natural thing to do is to
replace {\sf Set}, by {\sf Monoid}, which has a tensor product, and define the semantics of $\otimes$ using this tensor product.

\item Now we have a semantics for $V^L$, via forcing. Starting from a higher order many sorted many valued theory $T$, one can implement the above 
strategy, to define a  quasi-topoi, which only  contains only an almost all subobject classifier. But we not need a full sub-object classifier,
to interpret $T$ in quasi-topoi, in fact we do not haave one. So the semantics of the Lukasiewicz conjunction 
is defined like in the case of forcing. In the special case when  $T$ is set theory, 
then the semantics defined by forcing, will provide a complete semantics 
for fuzzy set theory. 

\end{enumarab}

\begin{theorem}
\begin{enumarab}
\item Let $T$ denote set theory formulated as a higher order many sorted theory. 
Let $E(T)$ be the category of definable terms and total functions. Then $E(T)$ is a topos. 
Let $H$ be the sub-object classifier of $E(T)$, and let $Set_H$ be the category of sheaves, that is functors from $H$ to set.
Give $H$ the truth values of the interpretation. Let $G$ be a translation function from $L(T)$ to the first order language of set theory.
Then $Set_H$ is naturally isomorphic to $E(T)$. Furthermore, given a formula $\phi$ in the many sorted language, then
the interpretation of $\phi[x]$ is true in $Set_H$, iff there exists $p\in H$,  such that $p$ forces $G(\phi)[[t]$, 
where $t$ is a term in the ramified language,
namely a sub-functor of a  sheaf on $C$.

\item Let $T$ denote fuzzy set theory formulated as a higher order many sorted theory.
Let $E(T)$ be the category of definable terms and total functions. Then $E(T)$ is a quasi-topos. 
Then there is an almost everywhere sub-object classifier, call it $A$, which is an $MV$
algebra. Let $A^*=A\times A$  
Give $A^*$ the truth values of $\otimes$ and $\land$.
Then $E(T)$ is equivalent to the category of fuzzy sets, which is a quasi-topos, namely the pre-sheaves
on $A$. Let $G$ be a translation as above. Then the interpretation of $\phi[x]$ is true in fuzzy $Set_A$ iff there exists $p\in A$,
such that $p$ forces $G(\phi)[t],$ where $t$ is a sub-functor of a pre-sheaf on $A$.
\end{enumarab}
\end{theorem}

\section{Many valued forcing}

Now let us produce many valued forcing, by modifying usual forcing.
Let $L$ be an $MV$ algebra, using $L$ one can define a many valued model, like in the Boolean case:
Inductively, on ordinals:
$V_0^L=\emptyset$, $V_{\alpha+1}^L=$ the set of all functions $x$ with $dom(x)\subseteq V_{\alpha}^H$ and values in $H$, 
and $V^L=\bigcup_{\alpha\in On}V_{\alpha}^H$.
In forming $V^L$, we change the the value of subset, hence equality in $L$. This reflects 
the fact that if we indentify the two conjuncts obtaining the Heyting algebra $H$, then $Set_H$ is only a refective subcategory of 
$Set_L$. The former consists of {\it only sheaves}, that is 'glued' functors from $H^{op}$ to ${\sf Set}$;
it is a topos, while the second consists of {\it all pre-sheaves}, which is just only a quasi topos.

Now to form $V^L$, the cumulative heirarchy of sets:
$$||x\in y||_L=\sum_{t\in dom(y)}(||x=t||\otimes y(t))$$
and
$$||x\subseteq y||_L=\prod_{t\in dom(x)}(x(t)\otimes(1\to  ||t\in y||),$$
$$||x=y||=||x\subseteq y||\otimes ||y\subseteq x||.$$
and the value of other formulas is obtained inductively as usual.
Our definition is also inspired by the fact that $\otimes$ and $\to$ are dual.
If $H$ is the Heyting algebra describled above then, that for $x,y\in V^L$, we have 
$$||x\in y||_L=||x\in y||_H$$ and conversely 
$$||x\subseteq y||_L\leq ||x\subseteq y||_H.$$

\begin{theorem} $V^{L}$ is a model of the $MV$ predicate logic, togother with extentionality, pairing, infinity, powerset, 
\end{theorem}
\begin{demo}{Proof} Like the standard proof. Let $a,b\in V^L$, let $c=\{a,b\}^B\in V^B$ such that $dom(c)=\{a,b\}$ and $c(a)=c(b)=1$. 
Then $||a\in c\land b\in c||=1$. This combined with separation gives the Pairing axiom.
We prove that for $X\in V^L$ there is $Y\in V^L$ such that
$||Y\subseteq X||=1$ and $||(\forall z\in X)(\phi(z)\leftrightarrow z\in Y)||=1.$
If $X\in V^L$, then letting $Y\in V^L$ as follows
$dom(Y)=\bigcup\{dom(u):u\in dom(X)\}$, $Y(t)=1$ for all $t\in dom(Y)$.
To prove power set. We show that for every $X\in V^B$, there is a $Y\in V^L$ such that
$||\forall u(u\subseteq X\to u\in Y)||=1$. Let $dom(Y)=\{u\in V^L: dom(u)=dom(X) \text { and } u(t)\leq X(t)\}$, $Y(u)=1$ for all $u\in dom(Y).$
Note that if $u\in V^L$ is arbitrary, let $u'\in V^L$ be such that $dom(u')=dom(X)$ and $u't)=X(t)\otimes ||t\in u||$ for all $t\in dom(X)$. Then
$$|u\subseteq X||\leq ||u=u'||.$$
We prove that for every $X\in V^L$ there is a $Y\in V^L$ such that
$$||(\forall u\in X)(\exists v\phi(u,v)\to \exists v\in Y\phi(u,v)||=1.$$ 
Let $dom(Y)=\bigcup\{S_u: u\in dom(X)\}$ 
$Y(t)=1$ for all $t\in dom(Y)$, where
$S_u\subseteq V^L$ is some set such that
$$\sum_{v\in V^L}||\phi(u,v)||=\sum_{v\in S_u}||\phi(u,v)||.$$
\end{demo}

Now every set (in $V$) has a canonical name in $V^L$ defined by
$\emptyset=\emptyset$ and  For every $x\in V$ let $\hat{x}\in V^L$ function with domain $\{\hat{y}: y\in x\}$ and for all $y\in x$, 
$\hat{x}(\hat{y})=1$.
Let $M$ be a countable transitive model of $ZFC$ and let $L$ be a an $MV$ algebra, 
Let $G$ be a maximal generic ultrafilter. Such ultrafilters exist,  by appeal to the Baire category theorem for Polish spaces, since 
$Spec(L)$ is a compact Hausdorff second countable space.
Then one can define the generic extension $M[G]$. 

\begin{definition}For every $x\in M^B$ 
we define $x^G$ by $\emptyset^G=\emptyset$ and $x^G=\{y^G: x(y)\in G\}$. We let $M[G]=\{x^G: x\in M^B\}$.
\end{definition}

\begin{theorem} Let $G$ be an $M$ generic ultrafilter on $L$. Then

(i) $x^G\in y^G\text { iff } ||x\in y||\in G$

(ii) $x^G=y^G \text { iff } ||x=y||\in G$
\end{theorem}
\begin{demo}{Proof} $$|x\in y|\in G\leftrightarrow \exists t\in dom(y)(y(t)\in G, |x=t|\in G)$$
$$\leftrightarrow \exists t(y(t)\in G x^G=t^G)$$
$$\leftrightarrow x^G\in \{t^G:y(t)\in G\}$$
$$\leftrightarrow x^G\in y^G$$

$$|x\subseteq y|\in G\leftrightarrow \prod_{t\in dom(x)}(x(t)\otimes (1\to |t\in y|)\in G$$
$$\leftrightarrow \forall t\in dom(x)(x(t)\in G \land 1\to |t\in y|\in G$$
(since $1\in G)$, then:)
$$\leftrightarrow \forall t(x(t)\in G\implies t^G\in y^G)$$
$$\longleftrightarrow x^G\subseteq y^G$$

\end{demo}

\begin{theorem} If $G$ is an $M$ generic ultrafilter on $B$, then
$$M[G]\models \phi(x_1,\ldots x_n)\Longleftrightarrow ||\phi(x_1,\ldots x_n)||\in G.$$
\end{theorem}

\begin{demo}The atomic formulas are dealt with above. The rest follows by induction using that $G$ is maximal.
\end{demo}

This notion of forcing satisfies the folowing forcing lemma for $p\in L$:
$$p\models \phi(a_1,\ldots a_n)\Longleftrightarrow \forall G(p\in G)(M[G]\models \phi(a_1^G,\ldots a_n^G).$$ 

Let $L$ be an $MV$ algebra. We can form a category $Set_L$ based on $L$, this category is not a topoi, though.
This category is equivalent to $V^L$, which is a functor category of of presheaves. Here unlike the Heyting case, sheaves
are {\it not enough}, in particular, it is not a {\it topoi}. 

The objects are pairs 
$(A,\alpha)$ where $\alpha:A\times A\to L$ is a map such that,
$\alpha(x,y)\leq \alpha(x,x)\land \alpha(y,y)$,  $\alpha(x,y)=\alpha(y,x)$ 
and $\alpha(x,y)\otimes(\alpha(y,y)\to \alpha(y,x)\leq \alpha(x,z)$. The morphisms between the objects $(A,\alpha)$ $(B,\alpha)$ are maps $f:A\to B$ 
such that
$(\forall x,y\in A)(\beta(f(x),f(y))\geq \alpha(x,y)$ and $(\forall x\in A)(\alpha(x,x)=\beta(f(x),f(x)))$. 
This category is complete but does not have a suboject classifier. 

Now, from $V^L$ we can obtain another (apparently different) category as follows. First we identify elements $u,v\in V^L$ for which $||u=v||=1$.
The objects of $Set^L$ are the identified objects and arrows are those identified $f\in V^L$ such that $||f\text { is a function }||=1$.

\begin{theorem} $Set^L$ and $Set_L$ are equivalent.
\end{theorem} 
\begin{demo}{Proof} With each $u\in V^{L}$ we associate the $H$ set $\bar{u}=(dom(u), \delta_u)$ where $\delta_u(x,y)=||x\in u\land x=y||_L$ 
and if $u,v$ and $f\in V^{L}$ are such that 
$V^{L}\models f:u\to v$, we obtain an arrow $\bar{f}:\bar{u}\to \bar{v}$ by defining 
$\bar{f}(x,y)=||f(x)=y||_L$. This yields an equivalence with inverse defined as follows. 
Let $(X,\delta)$ be an $L$ set. For each $x\in X$ define $\dot{x}\in V^L$ by $dom(\dot{x})=\{\hat{z}: z\in X\}$ and $\dot{x}(\bar{z}=\delta(x,z)$. 
Define $X^*\in V^L$ 
by $dom(X^*)=\{\dot{x}: x\in X\}$ and $X^*(\dot{x})=\delta(x,x)$. 
Given an arrow $f:(X,\delta)\to (Y,\epsilon)$ 
in $Set_L$ define $f^*\in V^L$ by $dom(f^*)=\{(\dot{x}, \dot{y}^(L): x\in X, y\in Y\}$ and $f^{*}(\dot{x}, \dot{y})^{(H)}=f(x,y)$.
\end{demo}

We show that in fuzzy set theory we cannot prove that Zorn's lemma proves the axiom of choice.

\begin{theorem} Let $\B$ be an $MV$ algebra. 
Then Zorn's lemma holds in $V^B$, but axiom of choice is independependent from $ZF$ based on $V$.
\end{theorem}
\begin{demo}{Proof} Let $\B$ be the gven algebra. Assume that
$$V^{B}\models (X,\leq_X)$$ 
is a non empty partially ordered set in which every chain has a supremum. 
Let $Y$ be the core of $X$ and define $\leq_Y$ on $Y$, by $y\leq_Yy'$ if
$||y\leq_X y'||=1$. Then $(Y,\leq_Y)$ is a partially ordered set in which every chain 
has an upper bound. So by Zorn's lemmm, which holds in $V$, $Y$ has a maximal element $c$.
We claim that
We claim that
$$||c \text { is a maximal element } of X||=1.$$
Take any $a\in V^B$ and define $V\in V^{B}$ by $dom(V)=dom(X)$ and 
$$V(x)=||x=a\land x\in X\land c\leq_X x||\lor ||x=c||.$$
Then $V^B\models V$ is a chain in $X$, so there is a $v\in Y$ such that
$$V^{H}\models v\text { is the supremum of } V.$$
Since $||c\in V||=1$, it follows that $||c\leq_X v||=1$ 
when $c\leq_Y v$ so that $v=c$ by maximality of $c$. This yield $$||a\in V\to a\leq_Xc||=1$$ and
$$||a\in V\to c\leq_X x||=1$$
Therefore
$$||a\in V\to a=c||=1.$$
Then
$$||a\in X\land c\leq_X a||\leq ||a\in V||.$$
Define the set $K\in V^{H}$ by $Dom(K)=\{\hat{p}:p\in K\}$ and $K(p)=O_p$. Then in $V^{H},$ $K$ is a subset of $\bar{P}$, and for $p\in P$,
$||\hat{p}\in K||=O_p$. 

Now take $P=N^{op}$, and the Heyting algebra $O(N^{op})$. Not every Heyting algebra is an $MV$ algebra 
but this one is, because it is linear. 
We will show that the $H$ valued set $K$ is infinite, but Dedekind finite. We have
$V^{H}\models K\subseteq \hat{N}$ and $V^{H}\models \neg n\in \hat{N}, n\in K$. But then in $V^H$ if for all $n\in N$, $n\in K$, 
then $K$ is not finite, so of $K$ were finite $\neg \forall n\in N$, $n\in K$, and so $\neg (\forall n\in N$ $n\in K)$.
We now dhow that in $V^H$ $K$ is Dedekind finite. If not, i.e if there existed an injection of $N$ into $K$, 
then the sentence $$\forall x\in K\exists y\in K x<y$$ would hold in $V^B$. But
$$||\forall x\in K\exists y\in K x<y||=\bigcap[O_m\implies \bigcup[O_n\cap ||m<n||]$$
$$=\bigcap[O_m\implies \bigcup_{m<n}O_n]$$
$$=\bigcap_{m}[O_m\implies O_{m+1}]$$
$$=\bigcap_{m}O_{m+1}=\emptyset.$$ 
\end{demo}
Using Zorn's lemma one can show that the ultrafilter theorem (every filter is contained in a maximal filter) 
holds in fuzzy  set theory. This means that all representation theorems in algebraic logic, like representability of locally finite algebras hold.

\section{Sheaf Duality and Epimorphisms}

For an algebra $\A$ and $X\subseteq \A$, $\Ig^{\A}X$ is the ideal generated by $X$. We write briefly lattice for a $BLO$; 
hopefully no confusion is likely to ensue.

\begin{definition}
\begin{enumarab}
\item  A lattice $L$ is regular if whenever $x$ is a prime ideal in $\Zd L$, then $\Ig^{L}\{x\}$ is a prime ideal in $L$.

\item A lattice $L$ is strongly regular, if whenever $x$ is a maximal ideal in $Zd \L$, then $\Ig^{I}\{x\}$ is a maximal ideal in $L$.

\item A lattice $L$ is congruence strongly regular, if whenever $x$ is a maximal ideal in $Zd\L$, then $\Co^{L}\{x\}$ is a maximal congruence of $L$.
\end{enumarab}
\end{definition}
If $L$ is not relatively complemented, then (2) and (3) above are not equivalent; but if it is relatively complemented then they are equivalent. 
A lattice with the property that every interval is complemented is called a relatively complemented lattice. 
In other words, a relatively complemented lattice is characterized by the property that for every element $a$ in an interval $[c,d]=\{x: c\leq x\leq d\}$
there is an element $b$, such that $a\lor b=d$ and $a\land b=c$. 
Such an element is called a complement; it may not be unique, but if the lattice is bounded then relative complements in $[a, 1]$ are just 
complements, and in case of distributivity such complements are 
unique.
In arbitrary lattices the lattice of ideas may not be isomorphic to the lattice of congruences, the 
following theorem gives a sufficient and necessary condition for this to hold.
The theorem is  a classic due to Gratzer and Schmidt.

\begin{theorem} For the correspondence between congruences and ideals to be an 
isomorphism it is necessary and sufficient that $L$ is distributive, relatively complemented with a minimum $0$.
\end{theorem}
\begin{proof} {\bf Sketch} Clearly the ideal corresponding to the identity relation is the $0$ ideal. 
Since every ideal of $L$ is a congruence class under some homomorphism, we obtain distributivity.
To show relative complementedness,  it suffices to show that if $b<a$, then $b$ has a complement in the interval $[0,a]$. 
Let $I_{a,b}$ be the ideal which consists of all $u$ with 
$u\equiv 0(\Theta_{a,b})$. $V_{a,b}$ is 
a congruence class under precisely one relation, hence $a\equiv b mod(\Theta[V_{a,b}])$.  Hence for some $v\in I_{a,b}$ we have 
$b\lor v=a$ and $b\land v=0$. 
Conversely, we have every ideal is a congruence class under at most one congruence relation, and of course under at least one.
\end{proof}

In case of relative complementation, we have
\begin{theorem} The following conditions are equivalent
\begin{enumarab}
\item $L$ is strongly regular
\item Every principal ideal of $L$ is generated by a an element in $\Zd L$
\item $\delta(L)$ is semisimple
\end{enumarab}
\end{theorem}
\begin{proof} Easy
\end{proof}

Our next example shows that semisimple algebras may not be regular, 
and stalks are not necessarily subdirectly indecomposable.

\begin{example}
Let $\C$ be a subdirectly indecomposable cylindric algebra of dimension $\alpha$. Let $I$ be the set of all finite subsets 
of subsets of $\alpha$. Let $F$ be an ultraflier on $I$ such that 
$X_{\Gamma}=\{\Delta\in I:\Gamma\subseteq \Delta\}\in F$ for all $\Gamma\in I$. Then the surjective homomorphism 
${}^I\C/ \to  {}^I \C/F$ induced by $F$ maps zero dimensional elements, to elements that are not zero dimensional.
\end{example}

We push the duality a step futher establishing a correspondence between open (closed) sets of $BLO$s and open subsets of its dual.
Recall that an ideal $I$ in $\A$ is regular if $\Ig^{\A}(I\cap \Zd\A)=I$.

\begin{theorem} There is an isomomorphism between the set of all regular ideals in $\Gamma(X, \delta)$ 
onto the lattice of open subsets of $X.$
\end{theorem}
\begin{proof} For $\sigma\in \Gamma(X,\delta),$ let $[\sigma]=\{x\in X: \sigma(x)\neq 0_x\}$. For $U\subseteq X$, let 
$J[U]=\{\sigma\in \Gamma(X, \delta): [\sigma]\subseteq U\}.$ 
Then $J\mapsto U[J]$ is an isomorphism, its inverse is $U[J]=\bigcup\{[\sigma]: \sigma\in J\}.$
\end{proof}

Note that a simple lattice is necessarily strongly regular (and hence regular), but the converse is not true, even in the case of strong regularity.
There are easy examples.

As an application to our duality theorem established above, we can show that certain properties can extend from simple structures to 
strongly regular ones, equivalently some properties that {\it do not} hold for strongly regular algebras, do not hold 
in the proper subclass of simple ones. In simple algebras the Stone space of $Zd\A$ is the two element Boolean algebra,
and any ideal in $\A$ is also just the two element Boolean algebra.
So trivially if one starts with an ideal in $\A$, restricts it to $\Zd\A$, and then lifts it to $\A$, then he gets back where he started.
In  strongly regular algebras, this happens with 
any {\it ideal} of $\A$.

The natural question that bears an answer is how far are strongly regular algebras from simple algebras; 
and the answer is: pretty  far. 
For example in cylindric algebras any non-complete theory $T$ in a first order language gives 
rise to a strongly regular $\omega$-dimensional algebra, namely, $\Fm_T$, that is not simple.
In addition, in localy finite algebras, {\it every} ideal is regular.(Note that here all notions of regularity coincide).


$ES$ abreviates that epimorphisms (in the categorial sense) are surjective. Such abstract property 
is equivalent to the well-known Beth definability property 
for many abstract logics, including fragments of first order logic, and multi-modal logics. 
In fact, it applies to any algebraisable logic (corresponding to a quasi-variety) regarded
as a concrete category.  This connection was established by N\'emeti. 
As an application, to our hitherto established duality, we have: 

\begin{theorem} Let $V$ be a class of distributive bounded lattices such that the simple lattices in $V$ 
have the amalgamation property $(AP)$.
Assume that there exist strongly regular $BLO$s $\A,\B\in V$ and an epimorphism $f:\A\to \B$ that is not onto.
Then $ES$ fails in the class of simple lattices.
\end{theorem}

\begin{demo}{Proof} Suppose, to the contrary that $ES$ holds for simple algebras.
Let $f^*:\A\to \B$ be the given epimorphism that is not onto.  We work in the dual space. We assume that $\A^d=(X,\L)$ and $\B^d=(Y,\G)$ 
are the corresponding dual sheaves over the Priestly  spaces $X$ and $Y$ and by  duality that 
$(h,k)=H:(Y,\G)\to (X,\L)$ is a monomorphism. Recall that $X$ is the set of prime ideals in $Zd\A$, and similarly for $Y$.
We shall first prove
\begin{enumroman}
\item $h$ is one to one
\item for each $y$ a maximal ideal in $\Zd\B$, $k(y,-)$ is a surjection of the stalk over $h(y)$ onto the stalk over $y$.
\end{enumroman}
Suppose that $h(x)=h(y)$ for some $x,y\in Y$. Then $\G_x$, $\G_y$ and $\L_{hx}$ are simple algebra, 
so there exists a simple $\D\in V$ and monomorphism $f_x:\G_x\to \D$ and $f_y:\G_y\to \D$ such that
$$f_x\circ k_x=f_y\circ k_y.$$
Here we are using that the algebras considered are strongly regular, and that the simple algebras have $AP$.
Consider the sheaf $(1,D)$ over the one point space $\{0\}=1$ and sheaf morphisms 
$H_x:(\lambda_x,\mu):(1,D)\to (Y,\G)$ and $H_y=(\lambda_y, v):(1,D)\to (Y,\G)$
where $\lambda_x(0)=x$ $\lambda_y(0)=y $ $\mu_0=f_x$ and $v_0=f_y$. The sheaf $(1,\D)$ is the space dual to $\D\in V$
and we have $H\circ H_x=H\circ H_y$. Since $H$ is a monomorphism $H_x=H_y$ that is $x=y$.
We have shown that $h$ is one to one.
Fix $x\in Y$. Since, we are assuming that  $ES$ holds for simple algebras of $V,$ in order to show that 
$k_x:\L_{hx}\to \G_x$ is onto, it suffices to show that $k_x$ is an epimorphism.  
Hence suppose that $f_0:\G_x\to \D$ and $f_1:\G_x\to \D$ for some simple $\D$ such that $f_0\circ k_x=f_1\circ k_x$.
Introduce sheaf morphisms 
$H_0:(\lambda,\mu):(1,\D)\to (Y,\G)$ and $H_1=(\lambda,v):(1,\D)\to (Y,\G)$
where $\lambda(0)=x$, $\mu_0=f_0$ and $v_0=f_1$. Then $H\circ H_0=H\circ H_1$, 
but $H$ is a monomorphism, so we have $H_0=H_1$ from which 
we infer that $f_0=f_1$. 

We now show that (i) and (ii) implies that $f^*$ is onto, which is a contradiction.
Let $\A^d=(X,\L)$ and $\B^d=(Y,\G)$. It suffices to show that $\Gamma((f^*)^d)$ is onto (Here we are taking a double dual) . 
So suppose $\sigma\in \Gamma(Y,\G)$. For each $x\in Y$, 
$k(x,-)$ is onto so $k(x,t)=\sigma(x)$ for some $t\in \L_{h(x)}$. That is $t=\tau_x(h(x))$ for some 
$\tau_x\in \Gamma(X,\G)$. Hence there is a clopen neighborhood $N_x$ of $x$ such that 
$\Gamma(f^*)^d)(\tau_x)(y)=\sigma(y)$ for all $y\in N_x$. 
Since $h$ is one to one and $X,Y$ are Boolean spaces, we get that $h(N_x)$ is clopen in $h(Y)$ and there is a 
clopen set $M_x$ in $X$ such that $h(N_x)=M_x\cap h(Y)$. Using compactness, there exists a partition of
$X$ into clopen subsets $M_0\ldots M_{k-1}$ and sections $\tau_i\in \Gamma(M_i,L)$ such that
$$k(y,\tau_i(h(y))=\sigma(y)$$ 
wherever $h(x)\in M_i$ for $i<k$. Defining 
$\tau$ by $\tau(z)=\tau_i(z)$ whenever $z\in M_i$ $i<k$, it follows that $\tau\in \Gamma(X,\L)$ and $\Gamma((f^*)^d)\tau=\sigma$.
Thus $\Gamma((f^*)^d)$ is onto $\Gamma(\B^d)$, and we are done. 
\end{demo}
And as an application, using known results, we readily obtain:

\begin{corollary}
\begin{enumarab}
\item Epimorphisms are not surjective in simple cylindric algebras, quasipolyadic algebras and Pinters algebras of infinite dimension
\item Epimorphisms are not surjecive in simple cylindric lattices of infinite dimension
\end{enumarab}
\end{corollary}

\begin{proof} (1) Cf. \cite{AUU} where two strongly regular algebras $\A\subseteq \B$ 
are constructed such that the inclusion is an epimorphism that is not 
surjective. 

(2) In a preprint of ours two strongly regular algebras $\A\subseteq \B$ 
are constructed, and the inclusion is not an epimorphism
\end{proof}

There is a very thin line between the superamalgamation $(SUPAP)$ and the strong amalgamation property $(SAP)$. 
However, Maksimova and Sagi- Shelah constructed varieties of $BAO$s with $SAP$ but not $SUPAP$, 
the latter is a variety of representable cylindric 
algebras. The second item of the next corollary makes one cross this line.

\begin{corollary}
\begin{enumarab}
\item  Let $V$ be a variety of $BAO$s such that every semisimple algebra is regular. Then if $ES$ holds for simple algebras, then it holds for 
semisimple algebras.
\item Let $V$ be a variety that has the strong amalgamation property, such that the simple algebras have $ES$. 
Then $V$ has the superamalgamation property.
\end{enumarab}
\end{corollary}

\begin{proof} We only prove the second part. If $SUPAP$ fails in $V$, then $ES$ does, because $V$ has $SAP$ 
and both together are equivalent to $SUPAP$, but then $ES$ fails in simple algebras 
and this is a contradiction.
\end{proof}

\end{document}